\documentclass[twoside, final]{article}

\usepackage[margin=1in]{geometry}

\title{Large Prandtl Number Asymptotics in Randomly Forced Turbulent Convection}
\author{Juraj F\" oldes, Nathan Glatt-Holtz, Geordie Richards\\
  \scriptsize{emails: foldes@virginia.edu, negh@tulane.edu, geordie.richards@usu.edu}}
\date{}

\usepackage{amsfonts, amssymb, amsmath, amsthm, multicol,bbm}
\usepackage[colorlinks=true, pdfstartview=FitV, linkcolor=blue,
            citecolor=blue, urlcolor=blue]{hyperref}
\usepackage[usenames]{color}
\definecolor{Red}{rgb}{0.7,0,0.1}
\definecolor{Green}{rgb}{0,0.7,0}
\usepackage{accents}
\usepackage{comment,url}

\usepackage[notref,notcite,color]{showkeys}
\definecolor{labelkey}{rgb}{0,0,1}

\numberwithin{equation}{section}

\newtheorem{Thm}{Theorem}[section]
\newtheorem{Lem}{Lemma}[section]
\newtheorem{Prop}{Proposition}[section]
\newtheorem{Cor}{Corollary}[section]

\newtheorem{Rmk}{Remark}[section]

\newtheorem*{Thm*}{Theorem}

\usepackage{cjhebrew}

\newcommand{\pd}[1]{\partial_{#1}}
\newcommand{\indFn}[1]{1 \! \! 1_{#1}}
\newcommand{\E}{\mathbb{E}}
\newcommand{\Prb}{\mathbb{P}}
\newcommand{\RR}{\mathbb{R}}

\newcommand{\bfU}{\mathbf{u}}
\newcommand{\bfV}{\mathbf{v}}
\newcommand{\bfW}{\mathbf{w}}
\newcommand{\bfY}{\mathbf{y}}
\newcommand{\bfX}{\mathbf{x}}
\newcommand{\T}{T}

\newcommand{\prN}{Pr}
\newcommand{\Nu}{Nu}
\newcommand{\ra}{Ra}
\newcommand{\hatk}{\hat{\mathbf{k}}}

\newcommand{\DD}{\mathcal{D}}
\newcommand{\Trs}{\theta}
\newcommand{\Thep}{\theta^{\varepsilon}}
\newcommand{\Thept}{\tilde{\theta}^{\varepsilon}}
\newcommand{\Tho}{\theta^{0}}

\newcommand{\Uep}{\mathbf{u}^{\varepsilon}}
\newcommand{\Uept}{\tilde{\mathbf{u}}^{\varepsilon}}
\newcommand{\Uo}{\mathbf{u}^{0}}
\newcommand{\uep}{u^\varepsilon}
\newcommand{\thep}{\theta^\varepsilon}

\newcommand{\eps}{\varepsilon}
\newcommand{\Ra}{\mbox{Ra}}

\newcommand{\phep}{\phi^\eps}

\newcommand{\Vep}{\mathbf{v}^{\varepsilon}}

\newcommand{\Vept}{\tilde{\mathbf{v}}^{\varepsilon}}

\newcommand{\phet}{\tilde{\phi}^{\eps}}
\newcommand{\rab}{\tilde{Ra}}

\newcommand{\JJ}{\mathcal{J}}
\newcommand{\AAA}{\mathcal{A}}
\newcommand{\MD}{\mathfrak{D}}
\newcommand{\HH}{\mathbf{H}}
\newcommand{\VV}{\mathbf{V}}

\usepackage[numbers,sort&compress]{natbib}

\newcommand{\smc}{\gimel}

\begin{document}
\markboth{J. F\" oldes, N. Glatt-Holtz, G. Richards}
{Large Prandtl Number Asymptotics in Randomly Forced Turbulent Convection}

\maketitle

\begin{abstract}
  We establish the convergence of statistically invariant states for
  the stochastic Boussinesq Equations in the infinite Prandtl number
  limit and in particular demonstrate the convergence of the Nusselt
  number (a measure of heat transport in the fluid).  This is a
  singular parameter limit significant in mantle convection and for
  gasses under high pressure.  The equations are subject to a both
  temperature gradient on the boundary and internal heating in the
  bulk driven by a stochastic, white in time, gaussian forcing.  Here,
  the stochastic source terms have a strong physical motivation for
  example as a model of radiogenic heating.

  Our approach uses mixing properties of the formal limit system to
  reduce the convergence of invariant states to an analysis of the finite time
  asymptotics of solutions and parameter-uniform moment bounds.  Here,
  it is notable that there is a phase space mismatch between the
  finite Prandtl system and the limit equation, and we implement
  methods to lift both finite and infinite time convergence results to
  an extended phase space which includes velocity fields.  For the
  infinite Prandtl stochastic Boussinesq equations, we show that the
  associated invariant measure is unique and that the dual Markovian
  dynamics are contractive in an appropriate Kantorovich-Wasserstein
  metric.  We then address the convergence of solutions on finite time
  intervals, which is still a singular perturbation.  In the process
  we derive well-posed equations which accurately approximate the
  dynamics up to the initial time when the Prandtl number is large.

\end{abstract}

{\noindent \small
  }

\setcounter{tocdepth}{1}
\tableofcontents

\newpage

\section{Introduction}

Buoyancy driven convection plays a central role in a wide variety of
physical processes: from earth's climate system to the internal
dynamics of stars.  As such it is of fundamental importance to
identify and predict robust statistical quantities in these complex
flows and to connect such statistics with the basic equations
governing their dynamics, for example the Boussinesq equations.  In
particular characterizing pattern formation, mean heat transport, and
small scale dynamics as a function of physical parameters and boundary
conditions remains a topic of intensive research theoretically,
numerically, and experimentally; see e.g.
\cite{BodenschatzPeschAhlers2000, Manneville2006,
  AhlersGrossmannLohse2009, LohseXia2010} for a broad overview of
recent developments.

It has long been understood that statistically invariant states of the
nonlinear partial differential equations of fluid dynamics provide
mathematical objects which are expected to contain various robust
statistical quantities found in turbulent fluid flows.  An ongoing
challenge is therefore to address the existence, uniqueness,
ergodicity, and dependence of these measures on parameters in a
variety of specific contexts.  While one may certainly pose such
questions for deterministic equations
(cf. \cite{FoiasManleyRosaTemam01}) the stochastic setting can be more
tractable given the regularizing effect of noise on the associated
probability distribution functions.  
Moreover, 
energy may be supplied to the system through both boundary or within 
the bulk of a fluid, the latter setting for instance models radioactive decay
processes in the earth's mantle; see
\cite{Roberts1967,TrittonZarraga1967,LuDoeringBusse2004,
  WhiteheadDoering2011, GoluskinSpiegel2012, BarlettaNield2012}. Both
sources can therefore have an essentially stochastic character in
situations of physical interest.

In this and a companion work,
\cite{FoldesGlattHoltzRichardsWhitehead2015}, we study statistically
invariant states of the stochastically driven Boussinesq equations
\begin{align}
  \frac{1}{\prN}&(\pd{t} \bfU + \bfU \cdot \nabla \bfU ) - \Delta \bfU 
   =  \nabla p  + \ra \hatk \T , \quad \nabla \cdot \bfU = 0,
  \label{eq:B:eqn:vel}\\
  &d \T + \bfU \cdot \nabla \T dt = \Delta \T dt +  \sum_{k=1}^{N} \sigma_k dW^k,
  \label{eq:B:eqn:temp}
\end{align}
for the (non-dimensionalized) velocity field $\bfU = (u_1,u_2,u_3)$,
pressure $p$, and temperature $T$ of a buoyancy driven fluid. The
system \eqref{eq:B:eqn:vel}--\eqref{eq:B:eqn:temp} evolves in a three
dimensional domain
$(x,y,z) = (\mathbf{x},z) \in\DD = [0,L]^{2} \times [0,1]$ and is
supplemented with the boundary conditions
\begin{align}
   \bfU_{|z = 0} = \bfU_{|z = 1} = 0, \quad T_{|z = 0} = \rab, \;  T_{|z = 1} = 0, \quad
   \bfU, T \textrm{ are periodic in } \mathbf{x} = (x, y).
   \label{eq:bc}
\end{align}
The unitless physical parameters in the problem are the \emph{Prandtl}
number $\prN$ and \emph{Rayleigh} numbers $\ra$ and $\rab$; see
Remark~\ref{rmk:Parameters:dep} and
\cite{FoldesGlattHoltzRichardsWhitehead2015} for further details
concerning this choice of nondimensionalization. The unit vector
$\hatk = (0,0,1)$ points in the direction of the gravitational force.
The driving noise in \eqref{eq:B:eqn:temp} is given by a collection of
independent white noise processes $dW^k = dW^k(t)$ acting spatially through 
the functions $\sigma_k = \sigma_k(x,y,z)$ which form a complete
orthogonal basis of eigenfunctions (ordered with respect to
eigenvalues) of the Laplace operator on $\DD$ supplemented with
homogeneous Dirichlet boundary conditions for $z = 0, 1$ and periodic
in $\mathbf{x}= (x,y)$.  The stochastic terms in \eqref{eq:B:eqn:temp}
have been normalized so that
\begin{align}
\label{cond:mode}
\sum_{k=1}^{N} \|\sigma_k\|_{L^2(\DD)}^2 = 1 \,,
\end{align}
with the strength of the body forcing expressed in terms of the
physical parameters $\ra$ and $\rab$; see
\eqref{eq:physical:parameters} below.

Our principal aim here is to establish convergence properties of
statistically invariant states of \eqref{eq:B:eqn:vel}--\eqref{eq:bc}
to invariant measures of the active scalar equation
\begin{align}
   &- \Delta \bfU =  \nabla p  + \ra \hatk \T , \quad \nabla \cdot \bfU = 0,
  \label{eq:B:vel:inf}\\
  &d \T + \bfU \cdot \nabla \T dt = \Delta \T dt +  \sum_{k=1}^{N} \sigma_k dW^k
  \label{eq:B:temp:inf}
\end{align}
in the \emph{Large Prandtl number limit} that is when $\prN$ in
\eqref{eq:B:eqn:vel} diverges to $\infty$.  Here
\eqref{eq:B:vel:inf}--\eqref{eq:B:temp:inf} is complemented with
boundary conditions as in \eqref{eq:bc}.  Note that $\bfU$ and $p$
are determined by $T$ according to \eqref{eq:B:vel:inf}. We
write this functional dependence as $\bfU = M(T)$ and denote $L(T)=(M(T),T)$.

The analysis of convection in the large Prandtl number limit is of
basic interest in a variety of physical contexts, most notably in
modeling certain portions of the earth's mantle and for convection in
gasses under high pressure, where the Prandtl number can reach the
order of $10^{24}$, see \cite{ConstantinDoerin1999,
  ConstantinDoering2001, OttoSeis2011}.  It is worth emphasizing that
the system \eqref{eq:B:vel:inf}--\eqref{eq:B:temp:inf} has very
complex dynamics even without stochastic forcing when the Rayleigh
number(s) are sufficiently large; see
\cite{BreuerHansen09,ConstantinDoerin1999, BodenschatzPeschAhlers2000,
  ConstantinDoering2001, Wang2004a, park2006,
  AhlersGrossmannLohse2009, LohseXia2010, OttoSeis2011}.

\subsection*{Overview of the Main Results}

Let us now present a heuristic version of our main results; for the
precise formulation see Theorem \ref{Thm:Main:Thm:Precise} below.
Recall that for any function $F$ and measure $\mu$, the push-forward
of $\mu$ under $F$ is given by $F\mu := \mu \circ F^{-1}$.

\begin{Thm} \label{thm:main:result:intro} Fix any $\ra, \rab >0$ and
  consider \eqref{eq:B:eqn:vel}--\eqref{eq:bc} and
  \eqref{eq:B:vel:inf}--\eqref{eq:B:temp:inf} with
  $N \geq N(\ra, \rab)$ independently forced directions in the
  temperature equation.  Then
  \eqref{eq:B:vel:inf}--\eqref{eq:B:temp:inf} possesses a unique
  ergodic invariant measure $\mu_\infty$.  Let
  $\{\mu_{\prN}\}_{\prN \geq 1}$ be any sequence of
  statistically invariant states associated to
  \eqref{eq:B:eqn:vel}--\eqref{eq:bc} satisfying a uniform exponential
  moment bound (see \eqref{eq:exp:yomoment:cond} below noting that
  $\epsilon : = \frac{1}{\Pr}$).\footnote{Note
    that usual fundamental difficulties concerning the well-posedness
    of 3D Navier-Stokes apply to
    (\ref{eq:B:eqn:vel})--(\ref{eq:B:eqn:temp}) and so, following
    \cite{FlandoliGatarek1}, we consider only weak
    solutions whose laws do not change in time.  The uniform
    exponential moment condition is analogous to a finite
    energy criterion for weak solutions of the 3D Navier-Stokes
    equations.  In \cite{FoldesGlattHoltzRichardsWhitehead2015} we
    have established the existence of such states $\mu_{\prN}$,  see
    Proposition~\ref{prop:existence:sol:BEs} below for a precise
    restatement.  In particular we cannot rule out the
    existence of a collection $\{\mu_{\prN}\}_{\prN \in \mathbb{N}}$ which does not satisfy 
    \eqref{eq:exp:yomoment:cond}.  Observe that none of these
    difficulties arise the 2D case.}  Then $\mu_{\prN}$ converges to
  $L_{\#}\mu_\infty$ in a suitable metric. In particular, for any
  sufficiently regular observable $\phi$ on the $(\bfU,T)$ phase
  space,
\begin{align}
 \left| \int \phi(\bfU,T) d \mu_{\prN} -  \int \phi(L(T)) d \mu_{\infty} \right| \leq C (\prN)^{-q} \,,
 \label{eq:conv:to:inf:prandlt:intro}
\end{align}
where $C = C(\phi, \ra, \rab), q = q(\ra,\rab) > 0$ are independent of $\prN$ and $q$ is independent of $\phi$.
\end{Thm}

\noindent The proof of Theorem~\ref{thm:main:result:intro} contains several
further results of independent interest.  Firstly, we show that the
Markovian dynamics of probability laws for the infinite Prandtl
system, \eqref{eq:B:vel:inf}--~\eqref{eq:B:temp:inf} is contractive in
a suitable Wasserstein distance; see Theorem~\ref{Thm:Main:Thm:Precise},
 \eqref{eq:ext:conv} below.  Secondly, we demonstrate that the finite
time dynamics converge in the limit as $\Pr \to \infty$.  

Note that our results 
do not rely on the well-posedness of \eqref{eq:B:eqn:vel}--\eqref{eq:B:eqn:temp}  
or make any assertions concerning the convergence of 
 \eqref{eq:B:eqn:vel}--\eqref{eq:B:eqn:temp} to the formal limit
\eqref{eq:B:vel:inf}--\eqref{eq:B:temp:inf} for small times.
On the other hand, as notable biproduct of our convergence analysis, we
derive a well posed approximation of
(\ref{eq:B:eqn:vel})--(\ref{eq:B:eqn:temp})
up to the initial time $t =0$ which is valid for large values of $\Pr$.
See Section~\ref{eq:corrector} and Theorem~\ref{lem:cor:to:LargePrandlt:og}
below for further details.

It is also worth emphasizing that our proof of Theorem
\ref{thm:main:result:intro} applies essentially verbatum to the
two-dimensional version of \eqref{eq:B:eqn:vel}--\eqref{eq:bc}, where
the horizontal variable $\mathbf{x}$ is one-dimensional.  Here
\emph{all} the statistically invariant states of the full system
satisfy the uniform moment bound \eqref{eq:exp:yomoment:cond}.
Furthermore, in collaboration with Whitehead
\cite{FoldesGlattHoltzRichardsWhitehead2015}, the authors have
established that with $N=\infty$ and $\prN = \prN(\ra,\rab)>0$
sufficiently large, the 2D version of
\eqref{eq:B:eqn:vel}--\eqref{eq:bc} possesses a unique ergodic
invariant measure $\mu_{\prN}$.

An empirical quantity of particular interest in convection is the
\textit{Nusselt number} $Nu$, a measurement of the convective heat
transfer, which is defined in terms of a statistical average (e.g. a
time average) of the observable
$\phi_{Nu}= \int_{\mathcal{D}}u_{2} T\, dx$.\footnote{Here $u_2$
  represents the vertical velocity component for the 2D version of
  \eqref{eq:B:eqn:vel}--\eqref{eq:bc}.}  However, in the
deterministic case, even in the turbulent regime of $\Ra \gg 1$,
$Nu$ depends on initial condition, both at finite
and infinite values of $\Pr$ and it is unclear that $Nu$
is continuous at $\prN = \infty$.   We show that
the addition of a stochastic perturbation avoids these concerns.

\begin{Cor} \label{cor:main}
For any $\ra, \rab >0$ and any sufficiently large $\prN = \prN(\ra,\rab)>0$ the system  \eqref{eq:B:eqn:vel}--\eqref{eq:bc} posed in two space dimensions
  with $N=\infty$ and \eqref{cond:mode}  possesses a
  unique ergodic invariant measure $\mu_{\prN}$, and the Nusselt
  number $Nu$ given by (cf. \cite[Theorem
  1.4]{FoldesGlattHoltzRichardsWhitehead2015})
\begin{align}
\label{eq:nu:unique}
\Nu = (\Nu)_{\prN} := 1 + \frac{1}{|\mathcal{D}|}\int \int_{\mathcal{D}}u_{2} T\, dx\,d\mu_{\prN}(\bfU,T)
\end{align}
satisfies
$$
\lim_{\prN \rightarrow \infty} (\Nu)_{\prN} = (\Nu)_{\infty} \,.
$$ 
Note that $(\Nu)_{\infty}$ is defined by \eqref{eq:nu:unique} relative to the unique ergodic invariant measure $\mu_{\infty}$ of
\eqref{eq:B:vel:inf}--\eqref{eq:B:temp:inf}.
\end{Cor}

Theorem~\ref{thm:main:result:intro} and Corollary \ref{cor:main} may
be seen as complementary to a series of recent works
\cite{Wang2004a,Wang2004b,Wang2005,Wang2007,Wang2008a,Wang2008b} which
address large Prandtl number asymptotics for the Boussinesq system in
a deterministic framework.  Here we show that the addition of
stochastic terms allows for stronger convergence results, but the
proofs require a different framework.  In particular, Corollary
\ref{cor:main} resolves a conjecture of Wang \cite{Wang2008b} by
confirming that stochastic forcing stabilizes the Nusselt number in
the infinite Prandtl number limit.

\subsection*{Methods of Analysis}

The starting point of our analysis is to establish a strict
contraction property for the Markov semigroup $\{P_t^0\}_{t \geq 0}$
associated to the formal limit system
\eqref{eq:B:vel:inf}--\eqref{eq:B:temp:inf}.  We show that for,
some $t_* > 0$ sufficiently large and for any probability measures
$\mu, \tilde{\mu}$ on the phase space associated with the $T$ component of
\eqref{eq:B:vel:inf}--\eqref{eq:B:temp:inf}, one has
\begin{align}
  \rho( \mu P_{t_*}^{0}, \tilde{\mu}P_{t_*}^{0} ) \leq \frac{1}{2} \rho(\mu, \tilde{\mu}),
  \label{eq:wash:approach:gap}
\end{align}
where $\rho$ is an appropriately chosen Kantorovich-Wasserstein
metric. See \eqref{eq:wash:dist:gen} and
Theorem~\ref{Thm:Main:Thm:Precise}, (i) for a precise formulation.  

The bound \eqref{eq:wash:approach:gap} is crucial since it allows us to reduce the proof of the
convergence of statistically invariant states in the infinite Prandtl
limit to the convergence of solutions on finite time intervals.
Indeed, suppose that $\mu_0$ is the (unique) 
invariant measure for (\ref{eq:B:vel:inf})--(\ref{eq:B:temp:inf})
and for $\eps > 0$ let $\mu_\eps$ be the $T$ component of any stationary solution
of (\ref{eq:B:eqn:vel})--(\ref{eq:B:eqn:temp}) with $\eps := 1/\Pr$.  Utilizing the
invariance of $\mu_0$ and \eqref{eq:wash:approach:gap}
we find
\begin{align}
\label{eq:strat:1}
  \rho(\mu_\eps, \mu_0) &=  \rho(\mu_\eps, \mu_0 P_{t_*}^{0})
  \leq \rho(\mu_\eps , \mu_\eps P_{t_*}^{0})
         +\rho(\mu_\eps P_{t_*}^{0}, \mu_0 P_{t_*}^{0}) 
  \leq \rho(\mu_\eps , \mu_\eps P_{t_*}^{0}) + \frac{1}{2} \rho(\mu_\eps, \mu_0)\,,
\end{align}
and consequently 
\begin{align*}
 \rho(\mu_\eps, \mu_0) \leq 2  \rho(\mu_\eps , \mu_\eps P_{t_*}^{0})  \,.
\end{align*}
By properties of the Wasserstein metric, specifically
(\ref{eq:eq:def:wd}), and using the stationarity of the solutions
corresponding to $\mu_\eps$ we therefore obtain the estimate
\begin{align}
\label{eq:strat:2}
\rho(\mu_\eps, \mu_0)  \leq 2\E\rho(T^{\eps}(t_*),T^{0,\eps}(t_*)).
\end{align}
We have thus bounded the distance between invariant states by the mean
distance between solutions at a fixed finite time $t_*$.  Note
that these two solutions satisfy identical 
initial conditions which are distributed as $\mu_\eps$.

Recently the strategy leading to \eqref{eq:strat:2} has proven
effective for establishing the convergence of statistically invariant
states for a variety of problems; see \cite{HairerMattingly2008,
  HairerMajda2010, KuksinShirikian12,
  FoldesFriedlanderGlattHoltzRichards2016, CerraiGlattHoltz2018}.
However, in order to implement this approach, one typically faces
several major challenges.  A first challenge is to prove the
contraction estimate \eqref{eq:wash:approach:gap}, where the
semigroup $\{P_t^0\}_{t \geq 0}$ corresponds to
\eqref{eq:B:vel:inf}--\eqref{eq:B:temp:inf}.  Moreover, in our
setting, it is desirable to lift this contraction property to the
extended phase space involving both the velocity $\bfU$ and
temperature components $T$ of our system.  This is particularly
relevant in view of the physical significance of the Nusselt number, a
quantity involving both $\bfU$ and $T$ as in \eqref{eq:nu:unique}.  A
second challenge  to show that
$\E\rho(T^{\eps}(t_*),T^{0,\eps}(t_*)) \to 0$ as $\eps \to 0$
in order to take advantage of \eqref{eq:strat:2}.  This task
requires suitable $\epsilon = \prN^{-1}$ uniform
moment bounds on the stationary statistics $\mu^\eps$ and 
finite time convergence results for solutions in the limit as
$\prN \to \infty$.  As we describe presently the results established
here require new ideas in comparison to the aforementioned related
works.  This is partially due to the presence of
non-homogeneous boundary conditions for
\eqref{eq:B:vel:inf}--\eqref{eq:B:temp:inf} and to the singular nature
of the limit from \eqref{eq:B:eqn:vel}--\eqref{eq:B:eqn:temp} to
\eqref{eq:B:vel:inf}--\eqref{eq:B:temp:inf}.

Regarding the first challenge, guided by the classical
Doob-Khasminskii Theorem
\cite{Doob1948,Khasminskii1960,ZabczykDaPrato1996} and as encompassed
by the more recent developments in \cite{HairerMattingly2008,
  HairerMajda2010,HairerMattinglyScheutzow2011}, one can establish a
contraction of the type \eqref{eq:wash:approach:gap} when the Markov
semigroup is smoothing, suitable moment bounds hold, and there is some
form of irreducibility in the dynamics.  The question of smoothing for
the Markov semigroup can be translated to a control problem; see
\eqref{eq:grad:system:2} below.  In our setting, when the number of
forced directions $N=N(\ra,\rab)$ is sufficiently large, an
appropriate control can be found through Foias-Prodi type
considerations \cite{FoiasProdi1967}. Since
\eqref{eq:B:vel:inf}--\eqref{eq:B:temp:inf} may be seen as an
advection diffusion system with $\bfU$ being two derivatives smoother
than $T$, such a strategy largely repeats the approach used in
previous works on the 2D stochastic Navier-Stokes equations
\cite{EMattinglySinai2001, HairerMattingly06, HairerMattingly2008,
  KuksinShirikian12}.  On the other hand establishing suitable moment
bounds is more delicate due to the non-homogenous boundary conditions
imposed in \eqref{eq:bc} and requires a careful use of the maximum
principle along with exponential martingale estimates.  These bounds
have been carried out in our companion work
\cite{FoldesGlattHoltzRichardsWhitehead2015}.  The main obstacle to
demonstrating \eqref{eq:wash:approach:gap} is to establish
irreducibility, which does not follow from the approach set out
in previous works, e.g. \cite{EMattingly2001,HairerMattingly06,
  ConstantinGlattHoltzVicol2013,FoldesGlattHoltzRichardsThomann2013,
  FoldesFriedlanderGlattHoltzRichards2016}.  This is because the
system \eqref{eq:B:vel:inf}--\eqref{eq:B:temp:inf} with its stochastic
terms removed can have highly non-trivial dynamics, see
\cite{ConstantinDoerin1999, ConstantinDoering2001, Wang2004a,
  park2006, OttoSeis2011}.  We show that, despite this complication,
the support of every invariant measure of
\eqref{eq:B:vel:inf}--\eqref{eq:B:temp:inf} contains the basic
conductive state.  Indeed we establish with the use of another
Foias-Prodi bound that a Girsanov shift of
\eqref{eq:B:vel:inf}--\eqref{eq:B:temp:inf} converges to the
conductive state with positive probability.  We then employ moment
estimates and stopping time arguments to translate this non-zero
probability back to the original system
\eqref{eq:B:vel:inf}--\eqref{eq:B:temp:inf} yielding the desired
irreducibility.

In order to establish convergence of invariant states on the extended
phase space, we adapt a methodology from recent joint work of the
authors with Friedlander
\cite{FoldesFriedlanderGlattHoltzRichards2016} which enhances
\eqref{eq:wash:approach:gap} to a ``lifted'' contraction property with
respect to a carefully chosen metric (see Lemma \ref{lem:lift} below).
By invoking this lifted contraction property, and appropriately
modifying the argument in \eqref{eq:strat:1}--\eqref{eq:strat:2}, the
convergence of invariant states as $\prN \rightarrow \infty$ reduces
to establishing the convergence of solutions of
\eqref{eq:B:eqn:vel}--\eqref{eq:B:eqn:temp} to those of
\eqref{eq:B:vel:inf}--\eqref{eq:B:temp:inf} at a fixed time $t_* > 0$,
independent of $\eps$, when the initial conditions have the same
distribution in temperature only.

The second major challenge regards the convergence of solutions of
\eqref{eq:B:eqn:vel}--\eqref{eq:B:eqn:temp} on finite time intervals
as $\prN \to \infty$ for which we develop a suitable asymptotic
analysis.  This is a non-trivial task since the small parameter
$1/\prN$ lies in front of the time derivative terms in
\eqref{eq:B:eqn:vel}.  Moreover the convergence analysis in
\cite{Wang2004a,Wang2004b,Wang2005,Wang2007,Wang2008a,Wang2008b} for a
deterministic analogue of \eqref{eq:B:eqn:vel}--\eqref{eq:B:eqn:temp}
requires significant modification.  In particular these references
crucially use higher temporal regularity properties which are missing
in our stochastic setting. As a
substitute we derive a stochastic evolution equation for the velocity
component and use martingale properties of associated It\=o integrals.
Our analysis then takes advantage of uniform moment estimates from
\cite{FoldesGlattHoltzRichardsWhitehead2015}, some previously
unobserved cancellations in certain error terms and delicate stopping
time arguments.

Analogous to the results in \cite{Wang2004a, Wang2004b, Wang2005,
  Wang2007, Wang2008a, Wang2008b} we derive an `intermediate system',
which we refer to as the `corrector'.  We show rigorously that this
system approximates the finite Prandtl system in the velocity equation
over bounded time intervals up the initial time; cf
Theorem~\ref{lem:cor:to:LargePrandlt:og}.  While this corrector system
is of independent interest we also we provide a somewhat simpler and
more direct analysis of the convergence of
\eqref{eq:B:eqn:vel}--\eqref{eq:B:eqn:temp} to
\eqref{eq:B:vel:inf}--\eqref{eq:B:temp:inf} which well approximates
the infinite Prandtl system after an $O(1/Pr)$ time transient. Indeed,
this more direct approach is sufficient for the upper bound in
\eqref{eq:strat:2} since this bound only involves a fixed time
$t_*> 0$.

\subsection*{Manuscript Organization}
The manuscript is organized as follows.  In
Section~\ref{sec:Math:setting} we introduce the rigorous mathematical
setting of the stochastic Boussinesq equations,
\eqref{eq:B:eqn:vel}--\eqref{eq:bc}, which serves as a foundation
for the rest of the analysis.  We also introduce the formalities
of the Kantorovich-Wasserstein metric in Section
\ref{sec:Wash:met:contraction}, and provide a rigorous formulation of
our main results in Section \ref{sec:main:res}.
Section~\ref{sec:reduction} describes core of our strategy that
reduces the question of convergence to finite time asymptotics
and uniform moment bounds.
Section~\ref{sec:mixing:inf:prandlt} is devoted to establishing the
contraction \eqref{eq:wash:approach:gap} for the infinite Prandtl
system \eqref{eq:B:vel:inf}--\eqref{eq:B:temp:inf}.  In
Section~\ref{sec:finite:time:conv} we carry out the finite time
convergence analysis.  The section concludes
with a derivation and analysis of the intermediate corrector system.
In Section \ref{sec:conv:nus}
we establish convergence of the Nusselt number.  Finally two  Appendices
recall various elements essentially contained in previous works that
we have used in our analysis.
Appendix~\ref{sec:mom:est:drift:diff:eqn} is devoted to
details for various moment estimates from
\cite{FoldesGlattHoltzRichardsWhitehead2015} for a class of
drift-diffusion equations which we use to bound
\eqref{eq:B:vel:inf}--\eqref{eq:B:temp:inf}.  In
Appendix~\ref{sec:grad:est:markov} we outline gradient estimates on the
Markov semigroup corresponding to
\eqref{eq:B:vel:inf}--\eqref{eq:B:temp:inf} which are carried out in a
similar fashion to e.g. \cite{HairerMattingly06}.

\section{Mathematical Preliminaries and Main Results}
\label{sec:Math:setting}

We begin our analysis of the stochastic Boussinesq Equations by
recalling some details of their mathematical setting.  The section
concludes with a mathematically precise restatement of
Theorem~\ref{thm:main:result:intro}. Here and below we implicitly
assume that $C, c, C_0$ etc. are constants depending on the domain
$\mathcal{D}$ with any other dependency indicated explicitly.

For the forthcoming analysis it is convenient to consider an
equivalent homogeneous, form of the stochastic Boussinesq Equations.
Introducing the `small parameter' $\eps = \prN^{-1} > 0$ and making
the change of variable $\theta^\eps = T - \rab(1 - z)$ we rewrite
\eqref{eq:B:eqn:vel}--\eqref{eq:B:eqn:temp} as\footnote{Note that we
  have implicitly modified the pressure in \eqref{eq:B:eqn:vel:mod} by
  $\ra\rab(z-\frac{1}{2}z^2)$ since
  $(1-z)\hatk=\nabla (z-\frac{1}{2}z^2)$.}
 \begin{align}
   &\eps(\pd{t} \Uep + \Uep \cdot \nabla \Uep) - \Delta \Uep
   = \nabla \tilde{p}^\eps  + \ra \hatk \Thep , \quad \nabla \cdot \Uep = 0,
   \label{eq:B:eqn:vel:mod}\\
 &d \Thep + \Uep \cdot \nabla \Thep dt = \rab \cdot \uep_3 dt + \Delta \Thep dt + \sum_{k =1}^{N} \sigma_k d  W^k,
\label{eq:B:theta:eps}
\end{align}
supplemented with the homogenous boundary conditions
\begin{align}
   \bfU^\eps_{|z = 0} = \bfU^\eps_{|z = 1} = 0, \quad \Thep_{|z = 0} =   \Thep_{|z = 1} = 0, \quad
   \bfU^\eps, \Thep \textrm{ are periodic in } \mathbf{x} = (x, y).
   \label{eq:bc:homo}
\end{align}
Here, in reference to the $\rab \cdot \uep_3$ term in \eqref{eq:B:theta:eps} recall that $\Uep = (\uep_1, \uep_2, \uep_3)$.
The corresponding infinite Prandtl system ($\eps = 0$) is given by
\begin{align}
 -\Delta \Uo &= \nabla \tilde{p} + \ra \hatk \Tho ,
 \quad \nabla \cdot \Uo = 0,
  \label{eq:vel:zero}\\
  d \Tho + \Uo \cdot \nabla \Tho dt &= \rab \cdot u^{0}_3 dt + \Delta \Tho dt + \sum_{k =1}^{N} \sigma_k dW^k,
    \label{eq:theta:zero}
\end{align}
again with initial conditions $\Tho(0) = \Tho_0$ and boundary conditions as in \eqref{eq:bc:homo}.
\begin{Rmk}
 Notice that we do not prescribe an initial condition for $\Uo$ in \eqref{eq:vel:zero}--\eqref{eq:theta:zero} as this 
 component does not satisfy an independent evolution equation.  Indeed \eqref{eq:vel:zero}-\eqref{eq:theta:zero}
can be rewritten as
\begin{align}
  d \Tho + (M \Tho) \cdot \nabla \Tho dt = \rab  (M \Tho)_3 dt + \Delta \Tho dt + \sum_{k =1}^{N} \sigma_k dW^k,
    \label{eq:theta:zero:Juraj:form}
\end{align}
where the constitutive law $M$ recovers $\bfU$ from $\theta$ according to
\eqref{eq:vel:zero} as in \eqref{eq:inf:pr:conts} below.
\end{Rmk}
\begin{Rmk}
  The systems \eqref{eq:B:eqn:vel:mod}--\eqref{eq:B:theta:eps} and
  \eqref{eq:vel:zero}-\eqref{eq:theta:zero} can be reformulated in
  terms of $T = \Thep + \rab(1-z)$, which satisfies
  \eqref{eq:B:eqn:vel}--\eqref{eq:B:eqn:temp} or
  \eqref{eq:B:vel:inf}--\eqref{eq:B:temp:inf}, respectively, and has boundary
  conditions \eqref{eq:bc}.  Our analysis makes use of
  both of these formulations.
\end{Rmk}
\begin{Rmk} \label{rmk:Parameters:dep}
As noted above, parameters in the problem are the Prandtl ($\prN = \eps^{-1}$) and Rayleigh numbers ($\ra$, $\rab$),
which are unit-less.   In terms of basic physical quantities of interest we have that 
\begin{align}
 \eps^{-1} = \prN = \frac{\nu}{\kappa}, \quad
 	\ra=  \frac{g \alpha  \gamma h^{5/2}}{\nu \kappa^{3/2}}, \quad
 	\rab = \frac{\sqrt{\kappa h} (T_b-T_t)}{\gamma},
	\label{eq:physical:parameters}
\end{align}
where $\nu$ is the kinematic viscosity, $\kappa$ the thermal
diffusivity, $g$ the gravitational constant, $\alpha$ the coefficient
of thermal expansion, $h$ the distance between the confining plates,
$T_b - T_t$ the temperature differential, and
$\gamma= \mathcal{H} /\rho c$ the intensity $\mathcal{H}$ of the
volumetric heat flux normalized by the density $\rho$ and specific
heat $c$ of the fluid.  We refer the interested reader to
\cite{FoldesGlattHoltzRichardsWhitehead2015}, where the dimensionless
form of the stochastically driven Boussinesq equations is derived.
\end{Rmk}

\subsection{Functional Setting of the Boussinesq Equations}
\label{sec:funct:setting}

We next define the phase space for the Boussinesq equations, which is
very close to the classical framework for the Navier-Stokes equations;
see e.g. \cite{ConstantinFoias88, Temam2001} for further details.

We define $\HH := H_1 \times H_2$
as the phase space for  \eqref{eq:B:eqn:vel:mod}--\eqref{eq:bc:homo},
where
\begin{align*}
  H_1 &:= \{ \bfU \in ( L^2(\DD))^3 : \nabla \cdot \bfU = 0, \bfU\cdot \mathbf{n}_{| z = 0,1} = 0, \bfU \textrm{ is periodic in $\bfX$}  \}, \\
  H_2 &:= \{  \Trs \in L^2(\DD) : \Trs \textrm{ is periodic in $\bfX$}  \}
\end{align*}
and we denote by $H = H_2$ the phase space for
\eqref{eq:vel:zero}--\eqref{eq:theta:zero}.  The spaces $\HH$ and $H$
are endowed with the standard $L^2$-norm and we denote each of them by
$\| \cdot \|$ as the appropriate meaning will be clear from the
context.\footnote{Below will also consider the weighted metrics
  \eqref{eq:metric:def}, \eqref{eq:rho:ext:PS} which generate an
  equivalent topology on $\HH$ and $H$ but are more suitable for the
  convergence of measures in the associated Wasserstein metric.}  All
other norms are written as $\| \cdot \|_{\mathbb{X}}$ below for a
given space $\mathbb{X}$.  We define $H^1$ type spaces as
\begin{align*}
  V_1 &:= \{ \bfU \in ( H^1(\DD))^3 : \nabla \cdot \bfU = 0, \bfU_{| z = 0,1} = 0, \bfU \textrm{ is periodic in $\bfX$}  \}, \\
  V_2 &:= \{  \Trs \in H^1(\DD) :\Trs_{| z = 0,1} =0,  \Trs \textrm{ is periodic in $\bfX$}  \}.
\end{align*}
Let $\VV= V_1 \times V_2$ and $V = V_2$.  We will sometimes consider
the $L^p(\mathcal{D})$ spaces of $p$-integrable functions for
$p \in [1, \infty]$ and endow these spaces with their standard norms.

In what follows we frequently project or lift the dynamics to account
for the phase space mismatch between
(\ref{eq:B:eqn:vel:mod})--(\ref{eq:B:theta:eps}) and
(\ref{eq:vel:zero})--(\ref{eq:theta:zero}).  We define
\begin{align}
  \Pi: \HH  \to H_2 \text{ to be the projection onto the $\Trs$ component of $\HH$.}
\label{eq:def:proj}
\end{align}
Associated with the limit system (\ref{eq:vel:zero})--(\ref{eq:theta:zero})
we have the constitutive law
\begin{align}
  M(\theta) = \ra A^{-1} P \theta \hat{k}
  \label{eq:inf:pr:conts}
\end{align}
where $A$ is the Stokes operator and $P$ the Leray projector. 
In other words $\bfU = M (\theta)$ is the solution of
\begin{align*}
-\Delta \bfU = \nabla \tilde{p} + \ra \hatk \theta, \quad \nabla \cdot \bfU = 0; 
\end{align*}
see Section~\ref{sec:stokes} and in particular \eqref{eq:stokes:steady} 
below.  We define the `lifting map' $L: H \to \HH$
from the temperature component to the extended phase space 
\begin{align}
  L (\theta) = (M(\theta), \theta).
  \label{eq:lifted:map}
\end{align}

Finally we denote $\Pr(\mathbb{X})$ as the space of Borel probability
measures on a given complete metrizable space $\mathbb{X}$, typically
$H, \HH$ etc.  For $\mu \in \Pr(\HH)$, we take
$\Pi \mu(\cdot) = \mu(\Pi^{-1}(\cdot))$ to be the push-forward of
$\mu$ by $\Pi$.  Similarly $L \mu$ is the push-forward of $\mu$ by $L$
when $\mu \in \Pr(H)$.

We have the following general results concerning the existence and uniqueness of 
solutions of  \eqref{eq:B:eqn:vel:mod}--\eqref{eq:bc:homo} and \eqref{eq:vel:zero}--\eqref{eq:theta:zero}:

\begin{Prop}[Existence, Uniqueness and Continuous Dependence of Solutions on Data]
\label{prop:existence:sol:BEs}
Fix any values $\Ra, \tilde{\Ra} > 0$.
\begin{itemize}
\item[(i)] For every $\epsilon > 0$ and any given $\mu^0 \in \Pr(\HH)$
  with
  $\int (\|\bfU\|^2 + \| \theta\|^2) d \mu^0(\bfU, \theta) < \infty$
  there exists a stochastic basis
  $\mathcal{S} = (\Omega, \mathcal{F}, \{\mathcal{F}_t\}_{t \geq 0},
  \Prb, W)$
  upon which is defined an $\HH$-valued stochastic process
  $(\bfU^\eps, \Trs^\eps)$ with the regularity
\begin{align*}
   (\bfU^\eps, \Trs^\eps) \in L^2(\Omega; L^2_{loc}([0,\infty); \VV) \cap L^\infty_{loc}([0,\infty); \HH)),
\end{align*}
which is weakly continuous in $\HH$, adapted to
$\{\mathcal{F}_t\}_{t \geq 0}$, satisfies
\eqref{eq:B:eqn:vel:mod}--\eqref{eq:B:theta:eps} weakly and such that
$ (\bfU^\eps(0), \Trs^\eps(0))$ is distributed as $\mu^0$.  We say
that such a pair $(\mathcal{S}, (\bfU^\eps, \theta^\eps))$ is \emph{a
  weak-martingale solution of
  \eqref{eq:B:eqn:vel:mod}--\eqref{eq:bc:homo}.}  If, for some
$p \geq 2$, and $\eta>0$,
 \begin{align}
    \int_{\HH} \exp(\eta ( \| \bfU \|^2 +  \|\theta\|^2_{L^p})) d\mu^0(\bfU, \theta) < \infty ,
    \label{eq:exp:yomoment:cond:IC}
  \end{align}
  there exists $\eta_0>0$ and a weak martingale solution $(\mathcal{S}, (\bfU^\eps, \theta^\eps))$ such that
    \begin{align}
    \E \exp\left(\eta_{0}\left( 
      \sup_{s \in [0,t]} (\|\bfU^\eps\|^2 + \|\Trs^\eps\|^2_{L^p})
        + \int_0^t (\| \nabla  \bfU^\eps \|^2 + \| \nabla \Trs^\eps \|^2) ds\right)
        \right) \leq C < \infty
    \label{eq:exp:yo:yo:m:c}
  \end{align}
  for each $t > 0$, where $C>0$ is a constant independent of $\eps\in (0,1]$.
  
\item[(ii)] Additionally, for any $\eps > 0$, there exists a
  martingale solution $(\mathcal{S}, (\bfU^\eps_S, \theta^\eps_S))$ of
  \eqref{eq:B:eqn:vel:mod}-\eqref{eq:B:theta:eps} which is stationary
  in time, meaning that the law of the solution is independent of
  time.  These stationary solutions
  $(\mathcal{S}, (\bfU^\eps_S, \theta^\eps_S))$ may be chosen in such
  a way that, for any $p \geq 2$ there is an
  $\eta = \eta(p, \Ra, \tilde{\Ra}) > 0$, for which
 \begin{align}
    \sup_{1 \geq \eps > 0} \int_{\HH} 
       \exp(\eta ( \| \bfU \|^2 +  \|\theta\|^2_{L^p})) 
   d\mu_\eps(\bfU, \theta) 
   =C_{0} < \infty ,
    \label{eq:exp:yomoment:cond}
  \end{align}
  where
  $\mu_\eps(\cdot) = \Prb( (\bfU^\eps_S(t), \Trs^\eps_S(t)) \in
  \cdot)$ for any fixed $t\geq 0$.
\item[(iii)] Now consider the case when $\eps = 0$.  Fix a stochastic
  basis $\mathcal{S}$ and any $\mathcal{F}_0$-measurable random
  variable $\theta_0 \in L^2(\Omega, H)$.  Then there exists a unique
  process $\theta^0$ with
\begin{align}
   \theta^0 \in L^2(\Omega; L^2_{loc}([0,\infty); V) \cap C([0,\infty); H)),
   \label{eq:esp:zero:regularity}
\end{align}
which is $\mathcal{F}_{t}$-adapted, weakly solves
\eqref{eq:vel:zero}--\eqref{eq:theta:zero}, and satisfies the initial
condition $\theta^0(0) = \theta_0$.
\item[(iv)] For a given stochastic basis $\mathcal{S}$ and each
  $\theta_0 \in H$ denote $\theta^0(\cdot, \theta_0, W)$ as the unique
  corresponding stochastic process satisfying
  \eqref{eq:vel:zero}--\eqref{eq:theta:zero} with
  \eqref{eq:esp:zero:regularity}.  We have that
  $\theta_0 \mapsto \theta^0(t, \theta_0, W)$ is Fr\'echet
  differentiable in $\theta_0 \in H$ for any $t \geq 0$ and any fixed
  realization $W(\cdot) = W(\cdot, \omega)$.  On the other hand
  $W \mapsto \theta^0(t, \theta_0, W)$ is Fr\'echet differentiable in
  $W$ from $C_0([0,t], \RR^N)$ to $H$ for each fixed $\theta_0 \in H$
  and $t > 0$.
\end{itemize}
\end{Prop}
These results are standard for a systems like
\eqref{eq:B:eqn:vel:mod}--\eqref{eq:bc:homo} and
\eqref{eq:vel:zero}--\eqref{eq:theta:zero}; see
e.g. \cite{ZabczykDaPrato1992,FlandoliGatarek1,
  DebusscheGlattHoltzTemam1, GlattHoltzHerzogMattingly2017}.  The only
novelty in view of existing methods is the uniform moment bound
\eqref{eq:exp:yomoment:cond}.  The existence of such a collection of
solutions is established using the maximum principle and exponential
moment bounds in the companion work
\cite{FoldesGlattHoltzRichardsWhitehead2015};
cf. Appendix~\ref{sec:mom:est:drift:diff:eqn} below.

The Markovian framework for \eqref{eq:vel:zero}--\eqref{eq:theta:zero} is defined as follows.  The transition functions are given by
\begin{align}
   P_t^0(\Trs_0, A) := \Prb( \Tho(t, \Trs_0) \in A), \quad t \geq 0, \; \Trs_0 \in H, A \in \mathcal{B}(H),
   \label{eq:trans:function}
\end{align}
where $\mathcal{B}(H)$ denotes the Borel sets of $H$, and the associated semigroup is given by
\begin{align} \label{eq:def:smg}
   P_t^0\phi(\Trs_0) :=  \E \phi(\Tho(t, \Trs_0)), \quad t \geq 0, \; \phi \in M_b(H),
\end{align}
where $M_b(H)$ is the set of bounded measurable functions on $H$.
In view of the continuous dependence on initial conditions the semigroup $\{P_t^0\}_{t \geq 0}$ is Feller, that is, it maps the set of
continuous bounded functions on $H$, $C_b(H)$, to itself.
This semigroup acts on Borelian probability measures $\mu$ according to
\begin{align}
\label{eq:mark:meas}
  \mu P_t^0(A)  = \int_H P_t^0(\Trs, A) d\mu(\Trs), \quad A \in \mathcal{B}(H).
\end{align}
A measure $\mu \in Pr(H)$ is said to be invariant with respect to $\{P_t^0\}_{t \geq 0}$ if
$\mu P_{t}^{0} = \mu$ for all $t\geq 0$.  Recall that in three space dimensions the Markovian framework for the full system with $\eps>0$ cannot be implemented due to a lack of global well-posedness. 

As an immediate consequence of bounds in Appendix~\ref{sec:mom:est:drift:diff:eqn}
and the Krylov-Bogolyubov averaging technique we have
\begin{Lem}\label{lem:ex:IM}
  Under the assumptions of Proposition~\ref{prop:existence:sol:BEs} and for any $\Ra, \tilde{\Ra}$
  there exists an invariant measure $\mu_0$ of the Markov semigroup $P_t^0$.  Moreover for any
  such measure
  \begin{align}
    \int_{H}\exp(\eta \| \theta \|^2_{L^p}) d \mu_0(\theta) \leq C_0 < \infty,
    \label{eq:est:IM:infpr}
  \end{align}
  for any $p \geq 2$ and any suitably small $\eta = \eta(p, \Ra, \tilde{\Ra})$.
\end{Lem}

\subsection{Wasserstein Distance, Weighted Metrics and Associated Observables}
\label{sec:Wash:met:contraction}

We next recall the general setting of the
Kantorovich-Wasserstein distance in which we establish our convergence
results.  We then introduce weighted metrics on $H$ and $\HH$
along with some associated classes of observable which are used to measure
distances between measures in the analysis below.  

Let $(\mathbb{X}, \rho)$ be a complete metric space
and take $\Pr_1(\mathbb{X}, \rho)$ to be the set of Borel probability measures $\mu$
on $\mathbb{X}$ with $\int \rho_{\eta} (0, \theta) d\mu(\theta) < \infty$.  On
$\Pr_1(\mathbb{X}, \rho)$ we define the Kantorovich-Wasserstein metric,
relative to $\rho$, equivalently as\footnote{Here note slight abuse of notation wherein
we denote both the underlying metric and its Wasserstein by $\rho$;  the meaning 
of $\rho$ will be clear from context in what follows.}
\begin{align}
   \rho (\mu, \tilde{\mu})
    :=  \sup_{\|\phi\|_{Lip, \rho} \leq 1} 
           \left| \int_\mathbb{X} \phi(\theta) d\mu(\theta) 
                  -  \int_\mathbb{X} \phi(\theta) d\tilde{\mu}(\theta)  \right|
    = \inf_{\Gamma \in \mathcal{C}(\mu, \tilde{\mu})}  
        \int_{\mathbb{X} \times \mathbb{X}}  \rho (\theta, \tilde{\theta}) d \Gamma(\theta, \tilde{\theta}),
    \label{eq:wash:dist:gen}
\end{align}
where
\begin{align}
\|\phi\|_{Lip, \rho} :=  \sup_{\substack{\theta \not =  \tilde{\theta}}} 
    \frac{|\phi(\theta) - \phi( \tilde{\theta}) | }{\rho(\theta, \tilde{\theta} )}
  \label{eq:lip:semi:nm}
\end{align}
for $\phi : \mathbb{X} \to \RR$, and $\mathcal{C}(\mu, \tilde{\mu})$
is the collection of Borel probability measures $\Gamma$ in
$Pr(\mathbb{X} \times \mathbb{X})$ with $\mu, \tilde{\mu}$ as its
marginals.  Hence, the last term in \eqref{eq:wash:dist:gen} is
equivalent to
\begin{align}\label{eq:eq:def:wd}
   \rho (\mu, \tilde{\mu}) = \inf \E \rho (X, Y) \,,
\end{align}
where the infimum is taken over all $\mathbb{X}$-valued random variables $X, Y$ distributed 
as $\mu, \tilde{\mu}$ respectively.   
See e.g. \cite{Villani2008, Dudley2002} for
further background on these metrics.

Specializing to our current setting, the following metrics on $H$ and $\HH$ 
prove useful for measuring the distance between the laws of solutions
of (\ref{eq:B:eqn:vel:mod})--(\ref{eq:B:theta:eps}) and 
(\ref{eq:vel:zero})--(\ref{eq:theta:zero}).
Following e.g. \cite{HairerMattingly2008} we introduce, for $\eta > 0$, the 
weighted metric on $H$ as
\begin{align}
\label{eq:metric:def}
   \rho_\eta (\theta, \tilde{\theta})
   = \inf_{\substack{\gamma \in C^{1}([0,1]; H)\\ \gamma(0) = \theta, \gamma(1) =\tilde{\theta} }}
   \int_0^1 \exp(\eta \|\gamma\|^2) \| \gamma'(s) \| ds, 
\end{align}
for any $\theta, \tilde{\theta} \in H$.
Notice that
\begin{align}
  \| \theta - \tilde{\theta}\| \leq \rho_\eta (\theta, \tilde{\theta})  \leq
  \exp(2\eta (\|\theta\|^2 + \|\tilde{\theta}\|^2))  \| \theta - \tilde{\theta}\|\,,
 \label{eq:norm:comp:1}
\end{align}
for $\theta, \tilde{\theta} \in H$.  For the extended phase space $\HH$,
similarly to our recent work \cite{FoldesFriedlanderGlattHoltzRichards2016},
we take
\begin{align}
  \tilde{\rho}_{\eta}((\mathbf{u},\theta),(\tilde{\mathbf{u}},\tilde{\theta})) 
  =  \|\mathbf{u}-\tilde{\mathbf{u}}\|_{H^1} + \rho_{\eta}(\theta,\tilde{\theta}),
  \label{eq:rho:ext:PS}
\end{align}
again defined for any $\eta > 0$.

For the statement of the main results we consider the following class of `observables'
\begin{align*}
    \mathcal{V}(\HH) = \mathcal{V}_{\eta}(\HH) 
  := \left\{ \phi \in C^1(\HH):     [\phi]_{\eta} < \infty \right\}\,,
\end{align*}
where the semi-norm $[ \cdot ]_{\eta}$ is given by
\begin{align*}
  [\phi]_{\eta} :=  
  &\sup_{(u,\theta) \in \HH} 
    \left[
   \sup_{\zeta\in \HH,\|\zeta \| =1} | \nabla _{u}\phi( u,\theta) \cdot\zeta | 
    +  \exp(-\eta\|\theta\|)\sup_{\xi\in H,\|\xi \| =1} | \nabla _{\theta}\phi( u,\theta) \cdot\xi |
    \right].
\end{align*}
Note that, as in \cite[Proposition 4.1]{HairerMattingly2008},
\begin{align}
  \|\phi \|_{Lip, \tilde{\rho}_\eta} \leq C [\phi]_\eta,
  \label{eq:norm:par:equiv}
\end{align}
for any $\phi \in C^1(\HH)$ with the constant $C$ independent of $\phi$.

\subsection{Statement of the Main Results}
\label{sec:main:res}

We now  precisely formulate the main results of this work 
on the convergence of solutions when $\Pr \to \infty$.  We begin
with the following finite time convergence result:
\begin{Thm}
\label{prop:conv:vel}
For each $\eps\in (0,1)$, let $(\Uep, \Thep)$ with its associated
stochastic basis $\mathcal{S}$ be a Martingale solution of
\eqref{eq:B:eqn:vel:mod}--\eqref{eq:B:theta:eps} in the sense of
Proposition~\ref{prop:existence:sol:BEs}.  Relative to this
$\mathcal{S}$, let $\Tho$ be a solution
of \eqref{eq:vel:zero}--\eqref{eq:theta:zero}.  Suppose there exists
$C_0, \eta> 0$ such that\footnote{Although we can relax the assumption
  on the initial velocity field to $q$th moment bounds for some
  $q \geq 4$, we have opted to impose an exponential moment condition
  for simplicity of presentation.}
\begin{align}
  \sup_{\eps > 0} \E \bigl[ \exp(\eta(\| \Uep(0) \|^2 
         +  \| \Thep(0) \|^2_{L^3}+ 
  \| \Tho(0) \|^2_{L^3})) \bigr] \leq C_0 < \infty,
  \label{eq:fin:tm:MB}
\end{align}
and suppose that $(\Uep, \Thep)$ maintains (\ref{eq:exp:yomoment:cond}).
Then, for each $t > 0$, there exists  $\gamma_0 >0$, $C> 0$ such that
\begin{align}
   \E&\left( \sup_{s \in [0,t]} \|  \Thep(s) - \Tho(s) \|^{p} 
        + \int_0^t  \|\Uep(s) - M(\Tho)(s)\|^{2}_{H^1}   ds \right)
       \label{eq:fin:time:conv}
  \\&\qquad \qquad \qquad \qquad \qquad
  \leq C \left(\eps^{\gamma} +
   \left(\E \| \Thep(0) - \Tho(0)\|^2 
        + \eps \E\|\Uep(0) - M(\Tho)(0)  \|^2 \right)^{\gamma} \right)
        , \notag
\end{align}
for each $0< \gamma \leq \gamma_0$ and any $p>\gamma$.  Here the constants
$C = C(p,\eta, \ra,\rab, C_0, \|\sigma\|_{L^3}, t)$ and
$\gamma_0 = \gamma_0(\eta, \ra,\rab, C_0, \|\sigma\|_{L^3}, t)$ are
independent of $\eps > 0$ and depends on the initial conditions only
through $C_0$.
\end{Thm}
\noindent The proof of Theorem~\ref{prop:conv:vel} is established in
Section~\ref{sec:f:t:conv:proof}.  

\begin{Rmk}
  It is worth noting that since $(\Uep, \Thep)$ are only Martingale
  solutions the associated stochastic bases $\mathcal{S}$ are not
  unique and could in fact vary as a function of $\eps$; that is, we
  cannot assume that these solutions are all defined relative to the
  \emph{same} stochastic basis.  Similar remarks apply to the bound
  \eqref{eq:ext:obs} below.  However, crucially, in both
  \eqref{eq:fin:time:conv} \eqref{eq:ext:obs} the constants do not
  depend on the choice of bases.  Thus, since this subtlety does not
  cause any trouble in what follows, we shall henceforth
  suppress this technical point in order to avoid notational confusion.
\end{Rmk}

We next state our results regarding the convergence of
statistically invariant states to the unique invariant measure of the
formal limit system (cf. Theorem~\ref{thm:main:result:intro}).
\begin{Thm}
   \label{Thm:Main:Thm:Precise}
   Let $\{P_t^{0}\}_{t \geq 0}$ be the Markov semigroup associated to
   \eqref{eq:vel:zero}--\eqref{eq:theta:zero} defined in
   \eqref{eq:def:smg}.  There exists $N_0 > 0$ and
   $\eta_0 > 0$ depending only on $\ra$ and $\rab$ such that
   if $N \geq N_0$, where $N_0$ is the number of stochastically forced modes in
   (\ref{eq:vel:zero})--(\ref{eq:theta:zero}), then the following bounds
   hold:
  \begin{itemize}
  \item[(i)] 
   For some $\gamma, C > 0$ depending only on $\ra$ and $\rab$
  \begin{align}
    \rho_{\eta}( \mu P_t^0 , \tilde{\mu}P_t^0) 
    \leq C \exp(- \gamma t)  \rho_{\eta}( \mu, \tilde{\mu}),
    \label{eq:conv:est}
  \end{align}
  for any $\mu, \tilde{\mu} \in \Pr_1(H, \rho_\eta)$,
  $\eta \in (0, \eta_0)$ and every $t \geq 0$, where $\rho_{\eta}$ is
  defined in \eqref{eq:wash:dist:gen}.  In particular, there exists a
  unique ergodic invariant measure $\mu_0 \in Pr_1(H, \rho_\eta)$ of
  \eqref{eq:vel:zero}--\eqref{eq:theta:zero}.
\item[(ii)] Suppose that $\{\mu_{\eps}\}_{\eps >0}$ is any collection
  of measures corresponding to stationary martingale solutions of
  \eqref{eq:B:eqn:vel:mod}--\eqref{eq:bc:homo} and satisfying the uniform
  bound \eqref{eq:exp:yomoment:cond} for any $\eta \in(0,\eta_0)$ and
  some $p \geq 3$. Let $\mu_0$ be the unique invariant measure of
  \eqref{eq:vel:zero}--\eqref{eq:theta:zero}.  Then, there exists
  $\tilde{q} = \tilde{q}(\ra,\rab)$,
  $\tilde{C} = \tilde{C}(\ra,\rab)$, independent of $\eps > 0$, such
  that
  \begin{align}
    \tilde{\rho}_{\eta}( \mu_{\eps}, L \mu_0)  \leq \tilde{C} \eps^{\tilde{q}}
    \label{eq:ext:conv}
  \end{align}
  for every $\eps > 0$. Consequently for the stationary processes 
  $(\mathbf{u}^{\eps}_{S},\theta^{\eps}_{S})$ and $\theta^{0}_{S}$, distributed
  as $\mu_{\eps}$ and $\mu_0$, respectively,
  \begin{align}
  |\E(\phi(\mathbf{u}^{\eps}_{S},\theta^{\eps}_{S}) - \phi(L\theta^{0}_{S}))| 
    \leq \tilde{C}[ \phi ]_{\eta}\eps^{\tilde{q}}
  \label{eq:ext:obs}
  \end{align}
  for any $\phi\in \mathcal{V}(\HH)$.
  \end{itemize}
\end{Thm}
\noindent The proof of (i) is carried out in
Section~\ref{sec:mixing:inf:prandlt} with some technical details
relegated to Appendices~\ref{sec:mom:est:drift:diff:eqn} and
\ref{sec:grad:est:markov}.  In Section~\ref{sec:reduction} we describe
a general strategy which shows that, under the conditions of
Theorem~\ref{Thm:Main:Thm:Precise}, \eqref{eq:ext:conv} follows from
\eqref{eq:conv:est} and \eqref{eq:fin:time:conv}.

We conclude this section by making several important remarks.
\begin{Rmk}
\label{rmk:main:theorem}
\mbox{}
\begin{itemize}
\item[(i)] Assertions of Theorem \ref{Thm:Main:Thm:Precise} also hold
  in two space dimensions and in addition one can show that
  \eqref{eq:B:eqn:vel:mod}--\eqref{eq:bc:homo} has a well defined
  Markov semigroup. Thus, any statistically invariant state
  corresponds to an invariant measure of the associated semigroup.
  This allows us to show in
  \cite{FoldesGlattHoltzRichardsWhitehead2015} that the
  $\eps$-independent exponential moment bounds in
  \eqref{eq:exp:yomoment:cond} hold for all invariant measures.
\item[(ii)] In 3D, the existence of a sequence of statistically
  invariant states of \eqref{eq:B:eqn:vel:mod}--\eqref{eq:bc:homo}
  satisfying the uniform moment bound \eqref{eq:exp:yomoment:cond} is
  established in the companion work
  \cite{FoldesGlattHoltzRichardsWhitehead2015}.  However we have not
  been able to show that \emph{every} sequence of invariant states
  have such (uniform) exponential moments.
\item[(iii)] In Section~\ref{sec:finite:time:conv} we also derive a
  'corrector' system which well approximates the dynamics of the
  velocity field of the full system
  (\ref{eq:B:eqn:vel:mod})--(\ref{eq:B:theta:eps}) up to the initial
  time for large values of $\Pr$ or equivalently small $\eps> 0$.  See
  Theorem~\ref{lem:cor:to:LargePrandlt:og} below for further details.
  Note however that this more refined version of
  \eqref{eq:fin:time:conv} is not needed in order to achieve
  (\ref{eq:ext:conv}), (\ref{eq:ext:obs}).
  \end{itemize}
\end{Rmk}

\section{Reduction to Finite Time Dynamics}
\label{sec:reduction}

In this section we describe a general strategy for reducing the
convergence of measures to finite time asymptotics and uniform moment
bounds when the formal limit system satisfies a suitable mixing
condition as in \eqref{eq:conv:est}.  To fix ideas we assume the
conditions of Theorem~\ref{Thm:Main:Thm:Precise} throughout this
section.  Also, we assume that both (\ref{eq:conv:est}),
(\ref{eq:fin:time:conv}) hold; we establish these bounds rigorously
below in Sections~\ref{sec:mixing:inf:prandlt},
\ref{sec:finite:time:conv} respectively.  The reader should note that
the presented method is flexible and can be applied in a variety of
settings.  See, for example, \cite{HairerMattingly2008,
  HairerMajda2010,KuksinShirikian12,
  FoldesFriedlanderGlattHoltzRichards2016}.

We adapt some ideas from our recent work 
\cite[Section 5]{FoldesFriedlanderGlattHoltzRichards2016} to the present 
setting.  For $\eta > 0$ take 
\begin{align*}
  \rho^*_\eta( \theta, \tilde{\theta})  
  = \tilde{\rho}_\eta(L(\theta), L(\tilde{\theta})) \,,
\end{align*}
where $\tilde{\rho}_\eta$ is defined in (\ref{eq:rho:ext:PS}) and
recall that $L$ is the lifting operator given in
(\ref{eq:lifted:map}).  It is not hard to show that the metrics
$\rho_\eta$ and $\rho_{\eta}^*$ are equivalent (see
\cite{FoldesFriedlanderGlattHoltzRichards2016} for details), and
consequently the associated Wasserstein metrics on $H$ are also
equivalent.  Then from (\ref{eq:conv:est}) we obtain the following
result, see \cite[Corollary
5.4]{FoldesFriedlanderGlattHoltzRichards2016} and surrounding
commentary for further details.

\begin{Lem}
\label{lem:lift}
  Under the same conditions as Theorem~\ref{Thm:Main:Thm:Precise} (i), we have
  \begin{align}
    \tilde{\rho}_\eta(L(\mu P_t^{0}), L(\tilde{\mu} P_t^{0}))
     \leq  C e^{-\gamma t}   \tilde{\rho}_\eta(L(\mu), L(\tilde{\mu}))
    \label{eq:lifted:cont}
  \end{align}
  for any $\mu, \tilde{\mu} \in \mbox{Pr}_1(H,\rho_\eta)$, 
  and every $t \geq 0$.  
 \end{Lem}

Using \eqref{eq:lifted:cont} choose $t^* > 0$  to guarantee 
that   
\begin{align}
  \tilde{\rho}_\eta(L(\mu P_{t^*}^0), L(\tilde{\mu} P_{t^*}^0)) 
  \leq \frac{1}{2} \tilde{\rho}_\eta(L(\mu), L(\tilde{\mu})).
  \label{eq:cont:t:cond}
\end{align}
By the invariance of $\mu_0$
\begin{align*}
   \tilde{\rho}_\eta( \tilde{\mu}, L \mu_0) 
   &\leq
      \tilde{\rho}_\eta( \tilde{\mu}, L( (\Pi \tilde{\mu})P_{t + t^*}^0))
       + \tilde{\rho}_\eta( L( (\Pi \tilde{\mu})P_{t + t^*}^0), L (\mu_0 P_{t + t^*}^0))\\
   &\leq
      \tilde{\rho}_\eta( \tilde{\mu}, L( (\Pi \tilde{\mu})P_{t + t^*}^0))
       +\frac{1}{2} \bigl[
          \tilde{\rho}_\eta( L( (\Pi \tilde{\mu})P_{t}^0),  \tilde{\mu}) + 
          \tilde{\rho}_\eta( \tilde{\mu}, L \mu_0 )
         \bigr]
\end{align*}
for any $t \geq 0$ and any other measure $\tilde{\mu} \in \Pr(\HH)$.
Here recall that $\Pi$ is the projection operator defined in (\ref{eq:def:proj}).
Rearranging, taking a time average we obtain
\begin{align}
  \tilde{\rho}_\eta( \tilde{\mu}, L \mu_0) 
  &\leq  \frac{2}{t^*}\int_0^{t^*} \bigl[
    \tilde{\rho}_\eta( \tilde{\mu}, L( (\Pi \tilde{\mu})P_{t + t^*}^0))
    +\tilde{\rho}_\eta( \tilde{\mu}, L( (\Pi \tilde{\mu})P_{t}^0)) 
    \bigr] dt 
  \notag\\
  &=  \frac{2}{t^*}\int_0^{2t^*}  
    \tilde{\rho}_\eta( \tilde{\mu}, L( (\Pi \tilde{\mu})P_{t}^0))
    dt \,. 
    \label{eq:ta:trck}
 \end{align}

 With \eqref{eq:ta:trck} now in hand, we consider an sequence of
 stationary Martingale solutions
 $\{(\bfU^\eps_S, \theta^\eps_S)\}_{\eps > 0}$ and take
 $\{\mu_\eps \}_{\eps > 0} \subset \Pr(\HH)$ to be the corresponding
 collection of stationary measures.  We suppose that
 $\{\mu_\eps \}_{\eps > 0}$ satisfies the uniform moment condition
 \eqref{eq:exp:yomoment:cond} as in
 Proposition~\ref{prop:existence:sol:BEs}, (ii).  We also denote
 $\theta_S^{0, \eps}$ (and $M(\theta_S^{0, \eps})$) the solution of
 \eqref{eq:vel:zero}, \eqref{eq:theta:zero} with the initial condition
 $ \theta^\eps(0)$ so that, for every $t > 0$, the law of
 $\theta_S^{0, \eps}(t)$ is $ (\Pi \mu_\eps)P_{t}^0$.  Consequently,
 \eqref{eq:eq:def:wd}, \eqref{eq:rho:ext:PS} yield
\begin{align}
\tilde{\rho}_\eta( \mu_\eps, L( (\Pi \mu_\eps)P_{t}^0)) 
  \leq \E \| \bfU^\eps_S(t) - M(\theta^{0,\eps}_S(t)) \|_{H^1} 
  + \E \rho_\eta ( \theta_S^\eps(t), \theta_S^{0, \eps}(t)),
  \label{eq:ex:ps:est:1}
\end{align}
where recall $M$ is defined as in (\ref{eq:inf:pr:conts}).
By \eqref{eq:norm:comp:1} one has, for any $q > 0$,
\begin{align*}
\E \rho_\eta ( \theta_S^\eps(t), \theta_S^{0, \eps}(t)) &\leq
\E \left( \exp( 2\eta (\| \theta_S^\eps(t) \|^2 + \| \theta_S^{0, \eps}(t) \|^2)) 
   \| \theta_S^\eps(t) - \theta_S^{0, \eps}(t) )\|\right)
  \notag\\
 &\leq C \E \left( \exp( 3\eta (\| \theta_S^\eps(t) \|^2 + \| \theta_S^{0, \eps}(t) \|^2)) 
   \| \theta_S^\eps(t) - \theta_S^{0, \eps}(t) \|^{q/2}\right) 
  \notag\\
 &\leq C \! \left(\E \exp( 12\eta \| \theta_S^\eps(t) \|^2) \cdot
 	\E \exp(12 \eta \| \theta_S^{0, \eps}(t) \|^2)) \right)^{1/4} \! \!
	\left(\E \| \theta_S^\eps(t) - \theta_S^{0, \eps}(t)  \|^q \right)^{1/2} 
   \! \! \! \! \! \! \,.
\end{align*}
Using \eqref{eq:gen:drift:diff:bnd:2} with $p = 2$ we obtain
\begin{align}
\E \rho_\eta ( \theta_S^\eps(t), \theta_S^{0, \eps}(t)) 
  \leq C \! \left(\E \exp( 96 \eta \| \theta_S^\eps(0)\|^2) \right)^{1/2} \! 
	\left(\E \| \theta_S^\eps(t) - \theta_S^{0, \eps}(t)  \|^q \right)^{1/2} 
   \! \! \! \! \! \! \,.
  \label{eq:ex:ps:est:2}
\end{align}
Finally combining \eqref{eq:ta:trck} with \eqref{eq:ex:ps:est:1},
\eqref{eq:ex:ps:est:2} we obtain
\begin{align}
\tilde{\rho}_\eta( \mu_\eps, L \mu_0) 
  \leq&  C \E\! \int_0^{2t*} \! \! \! \| \bfU^\eps_S(t) - M(\theta^{0,\eps}_S(t)) \|_{H^1} dt 
  \notag\\
      &+ 
  C \left( \E \exp( 96\eta \| \theta^\eps_S(0) \|^2)\right)^{1/2} \! \!
  \sup_{t \in [0,2t^*]}\left( \E \| \theta^\eps_S(t) -\theta^{0,\eps}_S(t)\|^q \right)^{1/2} \,,
  \label{eq:redux:bnd}
\end{align}
which holds for any $t^* > 0$ such that (\ref{eq:cont:t:cond}) holds.

With \eqref{eq:redux:bnd} established we conclude this section by
detailing the proof of Theorem~\ref{Thm:Main:Thm:Precise}, (ii) up to
the supporting results proven in Sections~\ref{sec:mixing:inf:prandlt}
and \ref{sec:finite:time:conv}.
\begin{proof}[Proof of Theorem~\ref{Thm:Main:Thm:Precise}, (ii)]
  The inequality \eqref{eq:conv:est} implies \eqref{eq:lifted:cont}
  which in turn implies the bound \eqref{eq:redux:bnd}.  Applying
  \eqref{eq:fin:time:conv} with
  $(\bfU^\eps, \theta^\eps) = (\bfU^\eps_S, \theta^\eps_S)$ and
  $\theta^0 = \theta^{0, \eps}_S$, noting
  $\theta^{0, \eps}_S (0) = \theta^\eps(0)$, and recalling the assumed
  bound \eqref{eq:exp:yomoment:cond} we infer \eqref{eq:ext:conv}.  To
  prove \eqref{eq:ext:obs}, let $C$ be as in \eqref{eq:norm:par:equiv}.
  Since the Lipschitz norm, with metric $\tilde{\rho}_\eta$
  of $\psi := \phi/C[\phi]_\eta$ is at most one, then, by
  \eqref{eq:wash:dist:gen}
  \begin{equation}
    |\E(\psi(\mathbf{u}^{\eps}_{S},\theta^{\eps}_{S}) - \psi(L\theta^{0}_{S}))|  = 
    \left| \int_{\HH} \psi(\bfU, \theta) \mu_\eps(\bfU, \theta) 
        - \int_{H} \psi( \bfU, \theta) L(\mu_0)(\theta, \bfU) \right| 
      \leq \tilde{\rho}_\eta(\mu_\eps, L\mu_0)\,,
  \end{equation}
  and the result follows from 
  \eqref{eq:ext:conv}.   The proof is complete.
\end{proof}

\section{Contraction in the Wasserstein Distance for the Infinite 
Prandtl System}
\label{sec:mixing:inf:prandlt}

In this section we establish some properties of the infinite Prandtl
system \eqref{eq:vel:zero}--\eqref{eq:theta:zero}, which provide a
sufficient condition for proving Theorem \ref{Thm:Main:Thm:Precise}
(i) as a consequence of a general result in \cite[Theorem
3.4]{HairerMattingly2008}. These properties are summarized as follows:
\begin{Prop}
\label{prop:conditions:for:gaps}
There exist $\eta_0 > 0$ and $N_0$, depending only on $ \Ra, \rab$,
such that for any $0 <\eta < \eta_0$, whenever the number of forced
modes $N$ exceeds $N_0$, we have
\begin{itemize}
\item[(a)] {\bf Lyapunov structure:} For all $t^{*} > 0$, there exists
  $C_1=C_1(t^{*},\eta)$ such that for each $\Tho_0 \in H$ and every
  $t \in [0, t^{*}]$,
\begin{align}
  \E \left(\exp( \eta \| \Tho(t, \Tho_0) \|^2) (1 + \|\JJ_{0, t}\|)\right) 
  \leq C_1 \exp( \eta (1+4\ra\rab)e^{-t/2} \|\Tho_0\|^2)\,,
  \label{eq:decay:exp:norm}
\end{align}
where the operator $\JJ_{0,t}$ is the Fr\'echet derivative of
$\theta^0(t, \theta_0)$ with respect to initial condition $\Tho_0$;
see \eqref{eq:J:op} and \eqref{eq:grad:system:0} below.
\item[(b)] { \bf Gradient Bound for Markov semigroup:}
for any $\phi \in C^{1}_b(H)$, and every $t \geq 0$, $\Trs \in H$
\begin{align}
   \|\nabla P_t^0 \phi(\Trs)\| \leq C \exp(\eta \| \Trs\|^2) 
  \left( \sqrt{ P_t^0 (|\phi(\Trs)|^2)} +  \delta(t) \sqrt{ P_t^0 (\|\nabla\phi(\Trs)\|^2)}\right),
   \label{eq:ASF:type:bnd}
\end{align}
where $\delta(t) \to 0$ as $t \to \infty$.  Here again
$\delta: [0,\infty) \to [0, \infty)$ and $C > 0$ depend only on
$\Ra, \rab$, and $\eta$.
\item[(c)] {\bf Irreducibility condition:}  for any $M  > 0$, $\epsilon > 0$ there is aF
  $t_{*} = t_{*}(M, \epsilon, \eta)$ such that for each $t \geq t_\ast$
\begin{align}
   \inf_{\| \Trs_{0}\|, \|\tilde{\Trs}_{0}\| \leq M}  \,
   \sup_{\Gamma \in \mathcal{C}(\delta_{\Trs_{0}} P_{t}^0, \delta_{\tilde{\Trs}_{0}}P_{t}^0)}
   \Gamma \{ (\Trs, \tilde{\Trs}) 
  \in H\times H: \rho_{\eta}(\Trs, \tilde{\Trs}) < \epsilon \} > 0,
   \label{eq:irreducibility:cond}
\end{align}
where, as above in \eqref{eq:wash:dist:gen},
$\mathcal{C}(\delta_{\Trs_{0}} P_{t}^0, \delta_{\tilde{\Trs}_{0}}
P_{t}^0)$
denotes the collection of all couplings of the measures
$\delta_{\Trs_{0}} P_{t}^0$ and $\delta_{\tilde{\Trs}_{0}} P_{t}^0$.
\end{itemize}
\end{Prop}

Proving the first item, (a), essentially reduces to establishing a
moment bound which follows from estimates found in
\cite{FoldesGlattHoltzRichardsWhitehead2015}, and which we recall
below in Appendix \ref{sec:mom:est:drift:diff:eqn} (see Proposition
\ref{prop:inf:prnd:bound}).  The second condition,
\eqref{eq:ASF:type:bnd}, can be translated to a control problem
through the use of Malliavin calculus which in our setting amounts to
proving a relatively straightforward Foias-Prodi type estimate.  Once
again (b) can be established by methods essentially contained in
previous works and we relegate further details to
Appendix~\ref{sec:grad:est:markov}.  As already mentioned above, the
principal novel challenge here is to prove the irreducibility
property (c) which we turn to next.

\subsection{Irreducibility}
\label{sec:Irred}

In previous related works the proof of irreducibility essentially
relies on the fact that the governing equations without the stochastic
forcing have a trivial attractor which is stable under small force
perturbations; see e.g.  \cite{EMattingly2001,HairerMattingly06,
  ConstantinGlattHoltzVicol2013, FoldesGlattHoltzRichardsThomann2013}.
In our present situation, \eqref{eq:vel:zero}--\eqref{eq:theta:zero},
the dynamics without body forces can be highly
non-trivial.\footnote{Note that the geometric control methods recently
  developed in \cite{GlattHoltzHerzogMattingly2017} by the second
  coauthor would be difficult to apply, as these methods seemingly require
  a detailed understanding of the wave-number interactions in
  (\ref{eq:vel:zero})--(\ref{eq:theta:zero}).}  Here our approach
reduces \eqref{eq:irreducibility:cond} to Foias-Prodi
type estimates through the Girsanov theorem and careful stopping time
arguments.  We believe that our strategy may be applicable to other
problems.

As a preliminary step we show that \eqref{eq:irreducibility:cond} follows from a simpler bound.

\begin{Lem}
For a given $N \geq 0$ consider \eqref{eq:vel:zero}--\eqref{eq:theta:zero} with $N$ independently forced directions.
If for every $M, \epsilon > 0$ there is a $t_\ast = t_\ast(M, \epsilon) > 0$ such that
\begin{align}
 \inf_{\|\Trs_{0}\| \leq M} \Prb ( \|\Trs^0(t, \Trs_{0})\| < \epsilon)  > 0, \textrm{ for each $t \geq t_\ast$}\,,
  \label{eq:irrd:desired}
\end{align}
then \eqref{eq:irreducibility:cond} holds for such an $N$ and any $\eta > 0$.
\end{Lem}

\begin{proof}
For any $\Trs_{0}, \tilde{\Trs}_{0} \in H$ consider the element $\tilde{\Gamma} 
\in \mathcal{C}(\delta_{\Trs_{0}} P_{t}^0, \delta_{\tilde{\Trs}_{0}}P_{t}^0)$ defined 
on cylindrical sets as
\begin{align*}
\tilde{\Gamma} (A \times B) = P_t(\Trs_0, A) \times P_t(\tilde{\Trs}_0, B),\quad A, B \in \mathcal{B}(H).
\end{align*}
For each $t > 0$  and any $M, \eta, \gamma > 0$ one has
\begin{align*}
     \inf_{ \|\Trs_0\|, \|\tilde{\Trs}_0\|  \leq M} \;&
     \sup_{\Gamma \in \mathcal{C}( \delta_{\Trs_0} P_t, \delta_{\tilde{\Trs}_0} P_t)}
     \Gamma\{ (\Trs, \tilde{\Trs}) \in H\times H: \rho_{\eta}(\Trs, \tilde{\Trs}) < \gamma \}  \notag\\
     &\geq     \inf_{ \|\Trs_0\|, \|\tilde{\Trs}_0\|  \leq M} \;
       \tilde{\Gamma} \left\{ (\Trs, \tilde{\Trs}) \in B_1\times B_1: \| \Trs\| 
       + \|\tilde{\Trs}\| < \gamma \exp(-4\eta)  \right\}  \\
      &\geq
    \left(\inf_{ \|\Trs_0\|  \leq M}  P_t
        \left(\Trs_0,  \Big\{ \Trs \in H:  \|\Trs\|  
               < \min \{\gamma/2 \cdot \exp(-4\eta), 1\} \Big\} \right) \right)^2 \\
    &= \left( \inf_{\|\Trs_{0}\| \leq M} \Prb ( \|\Trs(t, \Trs_{0})\| < \min \{\gamma/2 \cdot \exp(-4\eta), 1\})\right)^2,
\end{align*}
where we have used \eqref{eq:norm:comp:1} in the first inequality.
Applying \eqref{eq:irrd:desired} with
$\epsilon = \min \{\gamma/2 \cdot \exp(-4\eta), 1\})$ and the given
$M >0$ yields the desired result.
\end{proof}

In order to establish \eqref{eq:irreducibility:cond} the rest of the
section is therefore devoted to 
\begin{Prop}
\label{prop:irred:simp}
There exists $N_0 = N_0(\ra, \rab)$ sufficiently
large (cf. \eqref{eq:N:cond:irr}) such that, for any $N \geq N_0$  and every
$M, \epsilon > 0$, there is a $t_\ast = t_\ast(M, \epsilon) > 0$
such that \eqref{eq:irrd:desired} is satisfied.
\end{Prop}

\begin{proof}[Proof of Proposition~\ref{prop:irred:simp}]
  We first establish the analogue of \eqref{eq:irrd:desired} for the
  modified system
\begin{align}
 &-\Delta \bar{\bfU} = \nabla \bar{p} + \ra \hatk \bar{\Trs} \,, \quad \nabla \cdot \bar{\bfU} = 0,
  \label{eq:vel:zero:dip:shift}\\
  &d \bar{\Trs} + \bar{\bfU} \cdot \nabla \bar{\Trs} dt
  = (\rab \cdot \bar{u}_d + \Delta \bar{\Trs} -
  \lambda_N P_N \bar{\Trs}) dt + \sum_{k =1}^{N} \sigma_k d W^k,
  \quad \bar{\Trs}(0) = \Trs_{0} \,,
    \label{eq:theta:zero:dip:shift}
\end{align}
 when $N$ is sufficiently large.\footnote{Here
recall that $P_N$ denotes the projection onto the first $N$ modes of $-\Delta$ (with
boundary conditions as in \eqref{eq:bc:homo}) and $\lambda_N$ is the corresponding largest
eigenvalue in this collection.}  As in (\ref{eq:B:eqn:vel:mod})--(\ref{eq:B:theta:eps}) 
we supplement \eqref{eq:vel:zero:dip:shift}--\eqref{eq:theta:zero:dip:shift}
with the homogeneous boundary conditions (\ref{eq:bc:homo}).
Denote $\psi := \bar{\Trs} - \sum_{k =1}^{N} \sigma_k W^k = \bar{\Trs} - \sigma W$ which
satisfies
\begin{align*}
  &\partial_{t} \psi + \bar{\bfU} \cdot \nabla \psi
  = \rab \cdot \bar{u}_d + \Delta \psi - \lambda_N P_N \psi
  + (\Delta \sigma W - \lambda_N P_N \sigma W - \bar{\bfU} \cdot \nabla (\sigma W)), \qquad \psi(0) = \Trs_0.
\end{align*}
Taking an inner product with $\psi$, using that $\bar{\bfU}$ is divergence free, the inverse Poincar\'e
inequality (see \eqref{eq:lbd:rbr} below) and the bound
\begin{equation}\label{eq:nub:est}
\|\nabla \bar{\bfU} \| \leq \ra \| \bar{\Trs} \| \leq \ra (\| \psi \| + \| \sigma W\| )
\end{equation}
which follows from \eqref{eq:vel:zero:dip:shift} we have
\begin{align*}
    \frac{1}{2}\frac{d}{dt} \| \psi \|^{2}  + \lambda_N \| \psi\|^2
   &\leq  ( \rab \| \bar{\bfU}\| +  \|\Delta \sigma W\| + \lambda_N \|\sigma W\| 
     + \| \bar{\bfU}\| \| \nabla \sigma W\|_{L^{\infty}} )\|\psi\| \\
   &\leq C( \rab \ra (\| \psi \| + \| \sigma W\| ) +  \|\Delta \sigma W\| 
     + \lambda_N \|\sigma W\| + \ra (\| \psi \| + \| \sigma W\| ) \| \nabla \sigma W\|_{L^{\infty}} )\|\psi\|.
\end{align*}
For any $t > 0$, let $\xi_t = \sup_{0 \leq s \leq t}\{ \|\psi(s)\| = 0\}$ with the convection that
the supremum of the empty set is zero.  Thus, for any $t > 0$,
on the interval $[\xi_t, t]$ it follows that
\begin{align}
    \frac{d}{dt} \| \psi \| + (\lambda_N - \ra( \rab + \| \nabla \sigma W\|_{L^{\infty}}))\| \psi \|
   \leq&  (\ra \rab + \lambda_N + \ra\| \nabla \sigma W\|_{L^{\infty}} )\|\sigma W\| + \|\Delta \sigma W\|.
   \label{eq:good:bound:shift:shift}
\end{align}
Next, we use the fact that, with positive probability,
each of $\| \sigma W\|, \|\nabla \sigma W\|, \|\Delta \sigma W\|$ stays
close to zero over finite time intervals.
For $\gamma > 0$, $t >0$, $N > 0$ consider the sets
\begin{align*}
\mathcal{X}_{\gamma, t,N} := \left\{\sup_{s \in [0,t]} \| \nabla \sigma W\|_{L^{\infty}} \leq 1,
	\sup_{s \in [0,t]}  \|\Delta \sigma W\| \leq \frac{\gamma}{2},
	\sup_{s \in [0,t]}\|\sigma  W\| \leq \gamma\left( \frac{1}{2(\ra \rab + \lambda_N + \ra )}  \wedge  1\right) \right\}.
\end{align*}
Since $\sigma$ is spatially smooth we infer from standard properties of Brownian motion that
$\Prb( \mathcal{X}_{\gamma, t, N}) > 0$ for any $\gamma > 0$, $t > 0$, and $N> 0$.  On the other hand,
on $\mathcal{X}_{\gamma, t, N}$ the differential inequality
\begin{align*}
    \frac{d}{dt} \| \psi \| + (\lambda_N - \ra( \rab + 1))\| \psi \|
   \leq  \gamma
\end{align*}
holds over the interval $[\xi_t,t]$.   

Hence, fixing $N_0$ sufficiently large, we have for any $N\geq N_0$, 
\begin{align}
	\lambda_N  \geq \max\{2 \ra( \rab + 1) ,1\},
\label{eq:N:cond:irr}
\end{align}
and we infer that on $\mathcal{X}_{\gamma, t, N}$,
\begin{align*}
  \| \bar{\theta}(t, \theta_0) \| \leq \| \psi(t) \| + \|\sigma W\|
  \leq 2 \gamma+  e^{-\lambda_N t/2} \| \theta_0\|,
\end{align*}
where note that $\| \psi(t) \| = 0$ on the set where $\xi_t > 0$.
Therefore, for a given $M > 0$, $\epsilon > 0$, by choosing
$\gamma = \epsilon/4$ and $t_\ast = t_\ast(M, \epsilon)$ such that
$e^{- \lambda_N t_\ast} M \leq \frac{\epsilon}{2}$, we have for any
$t \geq t_\ast$
\begin{align}
  \inf_{\| \theta_0\| \leq M} \Prb(\| \bar{\theta}(t, \theta_0) \| < \epsilon) 
  \geq \Prb( \mathcal{X}_{\epsilon/4, t,N}) > 0 \,.
  \label{eq:irrd:desired:shf}
\end{align}

In order to now infer \eqref{eq:irrd:desired} from
\eqref{eq:irrd:desired:shf} we apply the Girsanov theorem and make
further bounds to a slightly modified version of
\eqref{eq:vel:zero:dip:shift}--\eqref{eq:theta:zero:dip:shift}.  For
$K > 0$ and $\theta_0 \in H$ define
$\tilde{\theta}_K = \tilde{\theta}_K(\cdot, \theta_0)$ as the solution
of \eqref{eq:vel:zero:dip:shift}--\eqref{eq:theta:zero:dip:shift} with
the term $-\lambda_N P_N \bar{\theta}$ replaced with
$-\lambda_N P_N \tilde{\theta}_K \chi_K( \| P_N \tilde{\theta}_K \|)$.
Here $\chi_K$ is a smooth, non-negative cut-off function with
$\chi_K \equiv 1$ for $|x| \leq K$ and $\chi_K \equiv 0$ for
$|x| \geq K +1$.  Consider the stopping times
\begin{align*}
  \tau_K(\theta_0)  = \inf_{s \geq 0} \left\{  \| P_N \tilde{\theta}_K(s, \theta_0) \| 
         \geq K \right\},
\end{align*}
for any $K > 0$ and any $\theta_0 \in H$.  It is not hard to see that
for any $K > 0$ and any $\theta_0 \in H$
\begin{align}
   \Prb \left( \bar{\theta}(t \wedge \tau_K(\theta_0), \theta_0) 
  =  \tilde{\theta}_K(t \wedge \tau_K(\theta_0), \theta_0), 
  \textrm{ for every } t \geq 0\right) =1.
   \label{eq:pathwise:unique:arg}
\end{align}
On the other hand, for any $\Trs_0 \in H$ and $K > 0$, the law of
$\tilde{\theta}_K(\cdot, \theta_0)$ is absolutely continuous with
respect to the law of the processes $\theta^0(\cdot, \theta_0)$
solving \eqref{eq:vel:zero}--\eqref{eq:theta:zero}.  Indeed, for
$\Trs_0 \in H$ and $K > 0$ define
\begin{align}
 \mathcal{M}_{\Trs_0,K}(t) = \exp\left( -\int_0^t \alpha_{\Trs_0,K} dW 
      - \frac{1}{2}\int_0^t |  \alpha_{\Trs_0,K} |^2 ds \right),
 \label{eq:gir:shift:def:trs:K}
\end{align}
where
\begin{align*}
\alpha_{\Trs_0,K}(s) = - \lambda_N \sigma^{-1} P_N \tilde{\Trs}_K(s,\Trs_0) 
  \chi_K( \| P_N \tilde{\theta}_K (s,\Trs_0) \|)
\end{align*}
and take
\begin{align*}
  d \mathbb{Q}_{\Trs_0,K,t} :=   \mathcal{M}_{\Trs_0, K}(t) d \Prb \,.
\end{align*}
Notice that, since
$| \sigma^{-1}P_N \tilde{\Trs}_K(s,\Trs_0) \chi_K( \| P_N
\tilde{\theta}_K \|) | \leq \|\sigma^{-1}\| \cdot (K+1)$,
the Novikov condition is satisfied and for any $\epsilon > 0$,
$t \geq 0$, $K > 0$, and $\theta \in H$, the Girsanov theorem yields
\begin{align*}
   \Prb ( \|\Trs(t, \Trs_{0})\| < \epsilon) 
      = \mathbb{Q}_{\Trs_0, K,t}(\|\tilde{\Trs}_{K} (t, \Trs_{0})\| < \epsilon)
      = \E \left( \mathcal{M}_{\Trs_0,K}(t)  
           \indFn{\|\tilde{\Trs}_{K} (t, \Trs_{0})\| < \epsilon} \right).
\end{align*}
Hence, for any $\epsilon>0$, $\theta_0 \in H$, and for any $\smc, K, t > 0$, the
Markov inequality implies
\begin{align*}
    \Prb ( \|\Trs(t, \Trs_{0})\| < \epsilon)    
  \geq \smc \Prb \left( \|\tilde{\Trs}_K(t, \Trs_{0})\| < \epsilon,   
                  \mathcal{M}_{\Trs_0,K}(t) \geq \smc \right)
        \geq \smc \Prb \left( \|\bar{\Trs}(t, \Trs_{0})\| < \epsilon,   
      \mathcal{M}_{\Trs_0,K}(t) \geq \smc, \tau_K(\theta_0) > t \right),
\end{align*}
where we used \eqref{eq:pathwise:unique:arg} for the final inequality.
On the other hand
\begin{align*}
\Prb \left( \|\bar{\Trs}(t, \Trs_{0})\| < \epsilon  \right)
  &\leq
  \Prb \left( \|\bar{\Trs}(t, \Trs_{0})\| < \epsilon, \mathcal{M}_{\Trs_0, K}(t) \geq \smc \right)
  +\Prb \left(  \mathcal{M}_{\Trs_0, K}(t) < \smc \right)\\
    &\leq
  \Prb \left( \|\bar{\Trs}(t, \Trs_{0})\| < \epsilon, \mathcal{M}_{\Trs_0,K}(t) 
      \geq \smc,  \tau_K(\theta_0) > t\right)
  +\Prb \left(  \mathcal{M}_{\Trs_0, K}(t) < \smc \right) +  \Prb(\tau_K(\theta_0) < t) \,.
\end{align*}
These two bounds yield
\begin{align}
\frac{1}{\smc}\inf_{\|\Trs_{0}\| \leq M} &\Prb ( \|\Trs(t, \Trs_{0})\| < \epsilon)    \notag\\
&\geq  \inf_{\|\Trs_{0}\| \leq M} \Prb \left( \|\bar{\Trs}(t, \Trs_{0})\| < \epsilon  \right) -
\sup_{\|\Trs_{0}\| \leq M} \biggl(\Prb \left(  \mathcal{M}_{\Trs_0,K}(t) < \smc \right) 
  +  \Prb(\tau_K(\Trs_0) < t) \biggr),
\label{eq:suff:cond:mom:irr:bnd}
\end{align}
for any $M, \epsilon, t> 0$ and for any $K, \smc > 0$.

Since the first term on the the right-hand side of
\eqref{eq:suff:cond:mom:irr:bnd} is independent of $K>0$ and
$\smc >0$, we finish the argument by showing that for every fixed
$M, K, t > 0$
\begin{align}
  \sup_{\|\Trs_{0}\| \leq M} \Prb \left(  \mathcal{M}_{\Trs_0, K}(t) < \smc \right) \to 0, 
  \quad \textrm{ as } \smc \to 0,
  \label{eq:Gir:shf:bnd}
\end{align}
and for every given $M, t > 0$
\begin{align}
\sup_{\|\Trs_{0}\| \leq M} \Prb(\tau_K(\Trs_0) < t)  \to 0, \quad \textrm{ as } K \to \infty.
\label{eq:stop:tm:K:bnd}
\end{align}
For the first bound \eqref{eq:Gir:shf:bnd}, we have from
\eqref{eq:gir:shift:def:trs:K} and It\=o isometry
\begin{align*}
   \Prb \left(  \mathcal{M}_{\Trs_0,K}(t) < \smc \right)
   &= \Prb \left(   \int_0^t \alpha_{\Trs_0,K} dW 
     + \frac{1}{2}\int_0^t | \alpha_{\Trs_0,K} |^2ds > \log(\smc^{-1}) \right) \notag\\
   &\leq \frac{1}{\log(\smc^{-1})} \E \left(\left| \int_0^t \alpha_{\Trs_0,K} dW \right| 
     + \frac{1}{2}\int_0^t | \alpha_{\Trs_0,K}|^2 ds \right)
   \notag\\
   &\leq \frac{2 }{\log(\smc^{-1})} \E \left( 1 
     + \lambda_N^2 \|\sigma^{-1} P_N\|^2 \int_{0}^{t} \|P_{N} \tilde{\Trs}(t,\Trs_0)\|^{2} 
     \chi_K ( \|P_N\tilde{\Trs}_K(t,\Trs_0)\|)ds\right),
   \notag\\
   &\leq \frac{2 \left( 1 + \lambda_N^2 \|\sigma^{-1} P_N\|^2 (K+1)^2 t\right)}{\log(\smc^{-1})} ,
\end{align*}
valid for any $\smc \in (0,1)$, $K >0$, and any $\Trs_0 \in H$.  For
the second bound, \eqref{eq:stop:tm:K:bnd} observe that, in view of
\eqref{eq:pathwise:unique:arg},
\begin{align}
 \Prb(\tau_K(\Trs_0) < t) \leq \Prb\left( \sup_{s \in [0,t]} \| P_N \bar{\Trs}(s, \Trs_0)\| \geq K \right)
 \leq \frac{1}{K^2} \E \left( \sup_{s \in [0,t]} \| \bar{\Trs}(s, \Trs_0)\|^2 \right)\,.
 \label{eq:stop:tm:K:bnd:arg:1}
\end{align}
From the It\={o} formula, and \eqref{cond:mode}, it follows that
\begin{align*}
   d \|  \bar{\Trs} \|^{2} +  2\lambda_N \|P_{N} \bar{\Trs}\|^{2}dt + 2\| \nabla \bar{\Trs}\|^{2} dt
   = \left(2 \rab \langle \tilde{u}_{d}, \bar{\Trs} \rangle +1\right)dt  + \langle\sigma, \bar{\Trs} \rangle dW.
\end{align*}
Integrating in time and using \eqref{eq:vel:zero:dip:shift}, inverse
Poincar\'e inequality (see \eqref{eq:lbd:rbr}), and \eqref{eq:nub:est} we infer
for any $s \geq 0$
\begin{align*}
   \| \bar{\Trs}(s)\|^2 +2 \lambda_N \int_{0}^{s} \|\bar{\Trs}\|^{2}dr \leq \| \Trs_{0}\|^{2} 
  + 2 \ra \rab \int_{0}^{s} \|\bar{\Trs}\|^{2}dr  +  s
    + 2\sup_{r \in [0,s]} \left| \int_0^r \langle\sigma, \bar{\Trs} \rangle dW  \right|\,.
\end{align*}
Using the assumption \eqref{eq:N:cond:irr} and the Birkholder-Davis-Gundy inequality we infer
\begin{align}
  \E \left( \sup_{s \in [0,t]} \| \bar{\Trs}(s, \Trs_0)\|^2 \right) \leq  \| \Trs_{0}\|^{2}+ 17 t.
  \label{eq:stop:tm:K:bnd:arg:2}
\end{align}
Combining \eqref{eq:stop:tm:K:bnd:arg:1} and
\eqref{eq:stop:tm:K:bnd:arg:2} thus yields the second bound
\eqref{eq:stop:tm:K:bnd}.

Using \eqref{eq:irrd:desired:shf}, \eqref{eq:suff:cond:mom:irr:bnd},
\eqref{eq:Gir:shf:bnd}, and \eqref{eq:stop:tm:K:bnd}, we conclude the
proof as follows.  Given any $\epsilon > 0$ and any $M >0$, and with
$\lambda_N$ given as in \eqref{eq:N:cond:irr}, choose $t_\ast$ as in
\eqref{eq:irrd:desired:shf}, that is, such that
$e^{- \lambda_N t_\ast} M \leq \frac{\epsilon}{2}$.  Fix any
$t \geq t_*$ and by \eqref{eq:irrd:desired:shf} we have
$a = a(M, \epsilon, t) := \inf_{\| \Trs_0\| \leq M} \Prb(
\|\bar{\Trs}(t, \Trs_0)\| < \epsilon) > 0$.
Now, by \eqref{eq:stop:tm:K:bnd}, we can pick $K$ sufficiently large so
that $\sup_{\|\Trs_{0}\| \leq M} \Prb(\tau_K(\Trs_0) < t) \leq a/4$,
and with $K, M, t$ fixed, we can by \eqref{eq:Gir:shf:bnd} choose
$\smc >0$ small enough so that
$\sup_{\|\Trs_{0}\| \leq M} \Prb \left( \mathcal{M}_{K,\Trs_0}(t) <
  \smc \right) \leq a/4$.
Finally, by combining these choices with
\eqref{eq:suff:cond:mom:irr:bnd} we obtain that
\begin{align*}
  \inf_{\| \Trs_0\| \leq M} \Prb( \|\Trs(t, \Trs_0) \| < \epsilon) \geq \frac{\smc a}{2} > 0 \,.
\end{align*}
The proof of Proposition~\ref{prop:irred:simp} is thus complete.
\end{proof}

\section{Finite Time Asymptotics}
\label{sec:finite:time:conv}

In this section we prove Theorem~\ref{prop:conv:vel}.  We also derive
a 'corrector' system which we show approximates the velocity component of
the full system \eqref{eq:B:eqn:vel:mod}--\eqref{eq:B:theta:eps} up to
the initial time (see Theorem~\ref{lem:cor:to:LargePrandlt:og} below).

\subsection{Preliminaries: The Stokes Operator}
\label{sec:stokes}
Before proceeding further we recall (see e.g. 
\cite{ConstantinFoias88, Temam2001}) 
some properties of solutions of the Stokes equation
\begin{align}
 - \Delta \bfU = \nabla p + \mathbf{f}, \quad \nabla \cdot \bfU = 0,
 \label{eq:stokes:steady}
\end{align}
supplemented with the mixed periodic-Dirichlet boundary conditions as in
\eqref{eq:bc:homo}.  We can express \eqref{eq:stokes:steady}
more abstractly as $A \bfU = P\mathbf{f}$, where $A = -P\Delta$ is the
Stokes operator.   Here, $P$ is the Leray
projection on divergence free vector fields $P:(L^2(\DD))^3 \to H_1$
with $H_1$, the space of $L^2$ divergence free vector fields,
defined in Section~\ref{sec:funct:setting}.\footnote{
Equivalently $A\bfU = - \Delta \bfU - \nabla p$, where
  $p = p(\bfU)$ the `pressure' is the unique $H^1$ function satisfying
  $\Delta p = -\mbox{div} (\Delta u)$ in the weak sense.}
As in the classical elliptic theory we have that
for any $\mathbf{f} \in (L^2(\DD))^3$, there exists a unique 
$\bfU \in D(A) = V_1 \cap (H^2(\DD))^3$ 
which satisfies
\begin{align}\label{eq:stk:reg}
   \| \bfU \|_{H^2} \leq  C \| \mathbf{f}\|, 
\end{align}
where $C$ is independent of $\mathbf{f}$.
  In what follows we
frequently denote $\bfU = A^{-1} P \mathbf{f}$.  

Since $A$ is a positive, self-adjoint operator which is unbounded on
the space $H_1$ with a compact inverse, by Hilbert's theorem there is
a complete orthonormal basis of eigenfunctions
$\{\mathbf{e}_k\}_{k \geq 1}$ of $A$ with the associated
non-decreasing sequence of eigenvalues $\lambda^{*}_k$ diverging to
infinity.   Take 
\begin{align}
  \text{$\mathcal{P}_N$ to be the projection onto the
  subspace $H_N := \mbox{span}\{\mathbf{e}_1, \ldots, \mathbf{e}_N\}$.}
  \label{eq:stnd:stk:proj}
\end{align}
Here the regularity theory as found in, say \cite{ConstantinFoias88, Temam2001},
show that each $\mathbf{e}_k$ is smooth and hence in particular $H_N \subset V$.

We also consider the associated linear evolution given as
\begin{align}
    \partial_t \bfU   - \mu \Delta \bfU = \nabla p + \mathbf{f}, 
  \quad \nabla \cdot \bfU = 0, \quad \bfU(0) = \bfU_0,
  \label{eq:stokes:linear:evolution}
\end{align}
for any parameter $\mu > 0$ and relative to the (sufficiently regular) data
$\mathbf{f}$, $\bfU_0$ supplemented with the boundary conditions \eqref{eq:bc:homo}.
Here, for any $\mathbf{f} \in L^2_{loc}([0,\infty); H_1)$ and $\bfU_0 \in H_1$
there exists a unique solution $\bfU$ of
\eqref{eq:stokes:linear:evolution} with $\bfU \in
L^2_{loc}([0,\infty);  V_1)\cap C([0,\infty); H_1)$.
Moreover, $A$ is the generator of an analytic semigroup which
we denote as $\{\exp(- \mu A t)\}_{t \geq 0}$.

\subsection{Finite Time Convergence Estimates}
\label{sec:f:t:conv:proof}

We next turn to the proof of Theorem~\ref{prop:conv:vel}:
\begin{proof}[Proof of Theorem~\ref{prop:conv:vel}]
  Take $\phep = \thep - \theta^0$ and $\Vep = \Uep - \bfU^0$ with
  $\bfU^0 = M (\theta^0)$, where $M$ is defined by
  \eqref{eq:inf:pr:conts}. Referring to
  \eqref{eq:B:eqn:vel:mod}--\eqref{eq:bc:homo} and
  \eqref{eq:vel:zero}--\eqref{eq:theta:zero} we see that $\phep$
  satisfies
\begin{align*}
	\pd{t}\phep - \Delta \phep &=
	 \rab \cdot v^{\eps}_{3} - \Vep \cdot \nabla \theta^0 - \Uep \cdot \nabla \phep,
	 \qquad \phep(0) = \thep(0) - \theta^0(0) := \phep_0 \,.
\end{align*}
Therefore, taking an $L^2$ inner product with $\phep$ and using that
$\nabla \cdot \Vep = 0$ we have
\begin{align*}
 \frac{1}{2} \frac{d}{dt} \| \phep\|^2 + \|\nabla \phep\|^2 
  &=  \int (\rab \cdot v^{\eps}_{d} - \Vep \cdot \nabla \theta^0) \phep dx
  \leq  \rab \|\Vep\|\| \phep\| +  \|\Vep\|_{L^6} \|\nabla \phep\| \| \theta^0\|_{L^3} \,.
\end{align*}
Hence from standard Sobolev embeddings, Young's inequality, and the
Poincar\'e inequality we obtain
\begin{align*}
   \frac{d}{dt} \| \phep\|^2  \leq C 
  \left( \| \theta^0 \|_{L^3}^2 + \rab^2 \right)\|\nabla \Vep\|^2 \,.
\end{align*}
Integrating in time we infer that
\begin{align}
  \sup_{s \in [0,t  ]} \| \phep(s \wedge \tau)  \|^2  \leq
  \|\phep_0\|^2 + 
   \sup_{s \in [0, t \wedge \tau ]} \left( \| \theta^0 (s) \|_{L^3}^2 
         + \rab^2 \right)\int_0^{t \wedge \tau} \|\nabla \Vep(t') \|^2 dt'
   \label{eq:phet}
\end{align}
for any $t >0$ and any stopping time $\tau \geq 0$.

We now turn to derive an evolution equation for $\Vep$. Recalling that
$\Uep$ and $\bfU^0$ satisfy respectively, \eqref{eq:B:eqn:vel:mod} and
\eqref{eq:vel:zero} 
we find 
\begin{align}
\eps(\pd{t} \Uep + \Uep \cdot \nabla \Uep) - \Delta \Vep  = \nabla q^\eps + \ra \hatk \phep,
\label{eq:vep:10}
\end{align}
where $q^\eps = \tilde{p}^\eps - \tilde{p}$ is the difference in the pressures.
On the other hand, recalling that $\bfU^0 = \ra A^{-1}( P(\hatk \theta^0))$,
we have
\begin{align}
   d \bfU^0 &=
	      -  \! \ra A^{-1} \!  P \! \left(\hatk \!\left( \bfU^0 \cdot \nabla\theta^0 \!
                                     - \Delta \theta^0 - \rab \cdot u_{3}^{0}\right) \! \right)  \! dt
	      + \ra \sum_{k =1}^N \! A^{-1} P(\hatk \sigma_k) dW^k\,.
	      \label{eq:vep:20}
\end{align}
Multiplying \eqref{eq:vep:10} by $\eps^{-1}$, subtracting the resulting system from \eqref{eq:vep:20}
and rearranging we obtain
\begin{align}
 d \Vep -& \frac{1}{\eps}\Delta \Vep  dt 
           = \frac{1}{\eps}\left(\nabla q^\eps  + \ra \hatk \phep \right) dt \notag\\
         &+ \!  \left(\ra A^{-1} \!  P \! \left( \hatk \! \left( \bfU^0 \cdot \nabla \theta^0 
           - \Delta \theta^0 + \rab \cdot u_{d}^{0}\right)  \! \right) \! -
  \Uep \cdot \nabla \Uep \right) \! dt - \ra \sum_{k =1}^N \! A^{-1} P(  \hatk \sigma_k) dW^k  \,,
  \label{eq:vept:evolution:eq:0}
\end{align}
with $\nabla \cdot \Vep = 0$.

Using \eqref{eq:vept:evolution:eq:0} we estimate $\Vep$ as follows.
The It\=o formula and \eqref{eq:vept:evolution:eq} yields
\begin{align}
  d\|\Vep\|^2  + \frac{2}{\eps} \|\nabla \Vep\|^2 dt
  =&
\frac{2}{\eps} \ra \langle \phep, v^{\eps}_d \rangle dt +
	2 \left\langle \ra  A^{-1}  P \! \left( \hatk \left( \bfU^0 \cdot \nabla\theta^0 - \Delta \theta^0
	- \rab \cdot u^{0}_{d}\right)  \right)-
	 \Uep \cdot \nabla \Uep, \Vep \right\rangle dt
  \notag\\
 &+ \ra^2 \sum_{k =1}^N | A^{-1} P( \hatk \sigma_k)|^2 dt
 - 2 \ra \sum_{k =1}^N \langle A^{-1} P(\hatk \sigma_k), \Vep \rangle dW^k
 \notag\\
 :=& (T_1 + T_2 + T_3 + T_4 + T_5 + T_6)dt + S dW.
 \label{eq:vept:evolution:eq}
 \end{align}
With the Young and Poincar\'e inequalities we have
\begin{align}
 |T_1| \leq \frac{1}{4 \eps} \|\nabla \Vep\|^2 + \frac{4\ra^2}{\eps} \|\phep \|^2.
 \label{eq:T1:bbd:og}
\end{align}
For $T_2$ we use that $A^{-1}$ is self-adjoint on $H$,
$D(A) \subset H$ and that $\bfU^0, \Vep$ are divergence free to obtain
\begin{align*}
|T_2| = 2 \ra \left|\int \bfU^0 \cdot \nabla \theta^0 (A^{-1}\Vep)_3 dx \right|
	= 2 \ra \left| \int \bfU^0 \cdot \nabla (A^{-1}\Vep)_3 \,  \theta^0  dx\right|
	\,,
\end{align*}
where $(A^{-1}\Vep)_3$ represents the third component of the vector field $A^{-1}\Vep$.
Hence \eqref{eq:stk:reg} and the imbedding $H^2 \hookrightarrow L^\infty$
imply
\begin{align}
|T_2| \leq 2 \ra \|\bfU^0\| \|\theta^0 \| \| \nabla (A^{-1}\Vep)\|_{L^\infty}
	\leq C \ra^2 \|\theta^0 \|^2 \| \nabla \Vep\|
	\leq \frac{1}{4 \eps}\| \nabla \Vep\|^2 + \eps C \ra^4  \|\theta^0 \|^4.
  \label{eq:T2:bbd:og}
\end{align}
For the terms $T_3$ and $T_4$ we use the regularity of the Stokes operator to obtain
\begin{align}
|T_3| + |T_4| 
  &\leq  \frac{1}{4 \eps} \|\nabla \Vep\|^2 
               + 4 \eps \ra^2 (\|\theta^0\|^2+ \rab^2 \|\bfU^0\|^2) \notag \\
  &\leq \frac{1}{4 \eps} \|\nabla \Vep\|^2 + \eps \ra^2 (\rab^2\ra^2 + 1) C \| \theta^0\|^2  \,.
 \label{eq:T34:bbd:og}
\end{align}

The most delicate term is $T_5$.   Here we take advantage of an additional cancellation 
to obtain extra regularity. Since $\Uep  = \Vep + \bfU^0$ we find
\begin{align}
  |T_5| &= 2 |\langle \Uep \cdot \nabla \bfU^0, \Vep \rangle|
          = 2 |\langle \Uep \cdot \nabla \Vep, \bfU^0 \rangle|
	\leq \frac{1}{4\eps} \|\nabla \Vep\|^2 +  4 \eps  \|\bfU^0\|^2_{L^\infty} \|\Uep\|^2
          \notag\\
       	&\leq \frac{1}{4\eps} \|\nabla \Vep\|^2 +  \eps  C\ra^2(\|\theta^0\|^4 +  \|\Uep\|^4),
      \label{eq:T5:bbd:og}
\end{align}
where we used the imbedding $H^2 \hookrightarrow L^\infty$ and \eqref{eq:stk:reg}
for the final bound.  Finally we observe $|T_6| \leq C \ra^2$.

Combining the bounds \eqref{eq:T1:bbd:og}--\eqref{eq:T5:bbd:og} and
rearranging in \eqref{eq:vept:evolution:eq} we find
\begin{align}
 d\|\Vep\|^2 \!  + \! \frac{1}{\eps}\|\nabla \Vep\|^2 dt
	&\leq \frac{4 \ra^2}{\eps} \|\phep \|^2 dt
	+  \eps C(1 + \ra^4)(1 + \rab^2) (\|\theta^0 \|^4  + \|\Uep\|^4 +1)dt  \notag\\
        &\quad + C \ra^2 dt +   2\ra \sum_{k =1}^N \langle A^{-1} P(\hatk \sigma_k), \Vep \rangle dW^k \,,
        \label{eq:vept:evo:est}
\end{align}
where the constant $C> 0$ is independent of $\ra, \rab$, and $\eps >0$.
Consequently, for any $t \geq 0$ and any stopping time $\tau$, we have
\begin{align}
  \int_0^{t \wedge \tau} \!\! \!\! \! \| \nabla \Vep\|^2 dt
	&\leq \eps \|\Vep(0)\|^2
       +  4 \ra^2 \int_0^{t \wedge \tau} \! \!\! \! \!(\|\phep \|^2 + \eps C) dt' 
	+ \eps^2 C(1+\Ra^4 )(1 + \rab^2)
	 \int_0^{t \wedge \tau} \! \! \!\! \!\left(  \|\theta^0\|^4 + \|\Uep\|^4 +1 \right)dt' 
          \notag\\
	&\quad -  \eps \ra \sum_{k =1}^N  \int_0^{t \wedge \tau} \langle A^{-1} P(\hatk \sigma_k), \Vep \rangle dW^k,
	\label{eq:vept:final:int:time:bnd}
\end{align}
where $C$ is independent of $\eps > 0$, $\ra, \rab$, and $\tau$.

Next for any $\kappa > 0$ define the stopping times
\begin{align}
  \tau_{\kappa} := \inf_{t \geq 0}\left\{\|\theta^0 (t)\|^2_{L^3} \geq \kappa \right\}\,.
    \label{eq:tau:kappa:stp}
\end{align}
From this definition and the bounds \eqref{eq:phet},
\eqref{eq:vept:final:int:time:bnd} we now infer
\begin{align*}
   \E \sup_{s \in [0,t  ]}& \| \phep(s \wedge \tau_\kappa)  \|^2   \notag\\
   &\leq \E \|\phep_0\|^2 +
  4 \ra^2 \left( \kappa + \rab^2 \right)
     \int_0^{t} \E \left(\sup_{s \in [0,t' ]} \| \phep(s \wedge \tau_\kappa)  \|^2\right)dt'
   + \eps \left( \kappa + \rab^2 \right) \bigl(\E\|\Vep(0)\|^2 + \ra^2 C t \bigr)
   \notag\\
   &\quad +	\eps^2 C ( \kappa + \rab^4 + 1) (\Ra^4 + 1) \int_0^{t }
	   \E   \left(   \|\theta^0\|^4 + \|\Uep\|^4 + 1\right)   dt',
\end{align*}
which implies with the Gronwall inequality that
\begin{align}
 \E \sup_{s \in [0,t  ]} \| \phep(s \wedge \tau_\kappa)  \|^2
	\leq \exp\left(  C (\ra^4 +1) \left( \kappa + \rab^4 + 1\right) (t+1)\right)
	         \left( \eps \mathcal{M}_\eps(t) + \E\| \phep_0 \|^2 \right),
	        \label{eq:phet:big:man}
\end{align}
where
\begin{align*}
  \mathcal{M}_\eps(t) := \E \| \Vep(0)\|^2
  +\int_0^t \left[ \eps \E \left(  \|\Uep\|^4 + \|\theta^0\|^4 \right)  + 1 \right] dt',
\end{align*}
and the constant $C$ is independent of $\kappa, \eps, \ra, \rab$, and $t$.
By \eqref{eq:fin:tm:MB} and our standing assumption that
\eqref{eq:exp:yo:yo:m:c} holds,
we observe that $\mathcal{M}_{\eps}$ is bounded independently of
$\eps > 0$ and we obtain
\begin{align}
\E \left(\sup_{s\in[0,t]}\|\phep(s)\|^{2}\indFn{\tau_\kappa > t}\right) \leq
 \E \sup_{s \in [0,t  ]} \| \phep(s \wedge \tau_\kappa)  \|^2
	\leq C_1 (\eps + \E\| \phep_0 \|^2 )  \exp\left(  C_{1}\kappa \right) ,
	        \label{eq:phet:big:man:2}
\end{align}
where the constant $C_1 = C(\ra, \rab, t)$ is independent of $\eps>0$ and $\kappa>0$.

Set $X_\eps(t) := \sup_{s \in [0,t]} \|\phep(s)\|^2$ and for each $t \geq 0$,
$\kappa > 0$, $\eps > 0$ define the sets
\begin{align*}
  E_{t, \kappa, \eps} := \left\{ \sup_{s \in [0,t]} \| \theta^0(s) \|^2_{L^3} 
                   \geq \kappa \right\} = \{\tau_{\kappa} \leq t\}.
\end{align*}
For each $t > 0$ one finds by the Markov inequality, Proposition
\ref{prop:lets:make:the:most:of:the:moments:that:count}, and the
assumption \eqref{eq:fin:tm:MB} that for sufficiently small
$\eta_1 =\eta_1(\rab, \eta) > 0$,
\begin{align}
\Prb ( E_{t, \kappa, \eps} )
  \leq e^{-\eta \kappa} \E\exp \left(\eta_1 \sup_{s \in [0, t]} \|\theta^0(s)\|^2_{L^3} \right)
  \leq C_2 e^{-\eta \kappa},
\label{eq:stop:smp}
\end{align}
where $C_2 = C_2(C_0, \ra, \rab, \|\sigma\|_{L^3}, \eta, t) > 0$ is independent of $\eps > 0$ and $\kappa > 0$.
On the other hand, for any $\gamma \in (0,1)$ we have
\begin{align}
  \E X_\eps(t)^\gamma  &=
    \sum_{k =0}^\infty \E \left(X_\eps(t)^\gamma 
     \indFn{k \leq \left(\sup\limits_{s \in [0,t]} \|\theta^0(s)\|^2_{L^3} \right)  < k + 1} \right)
  = \sum_{k =0}^\infty \E \left(X_\eps(t)^\gamma \indFn{\tau_k \leq t} \indFn{\tau_{k+1} > t} \right) \notag\\
  &\leq \sum_{k =0}^\infty \left( \E(X_\eps(t) \indFn{\tau_{k+1} > t} ) \right)^\gamma \left( \Prb( E_{t, \kappa, \eps}) \right)^{1-\gamma}
  \notag\\
  &\leq  C (\eps + \E\| \phep_0 \|^2 )^{\gamma}
  \sum_{k =0}^\infty  \exp(\gamma C_1 (k+1) -(1- \gamma) \eta k) \,,
  \label{eq:re:sum:trick}
\end{align}
where we have used \eqref{eq:phet:big:man:2} and \eqref{eq:stop:smp}
for the final bound.  Here,
$C = C(C_0, \ra, \rab, \|\sigma\|_{L^3}, \eta, t)$ is independent of
$\eps > 0$, $\E \| \phi^\eps_0\|^2$ and $C_1$ is the constant
appearing in \eqref{eq:phet:big:man:2}.  Thus when
$\gamma < \frac{\eta}{C_1 + \eta}$ the series in
\eqref{eq:re:sum:trick} converges.  Then for any $p>0$ and any
$\gamma < (\frac{\eta}{C_1 + \eta})\wedge p$ we find
\begin{align*}
  \E \sup_{s \in [0,t]} \|\phep(s)\|^p
  \leq C \left(\E \sup_{s \in [0,t]} (\|\Thep(s)\|^{2(p- \gamma)} + \|\Tho(s)\|^{2(p- \gamma)})\right)^{1/2}
  \left(\E \sup_{s \in [0,t]} \|\phep(s)\|^{2\gamma}\right)^{1/2}.
\end{align*}
Combing this bound with \eqref{eq:fin:tm:MB} and \eqref{eq:re:sum:trick}
we now obtain the first part of \eqref{eq:fin:time:conv}.

To address the second term in \eqref{eq:fin:time:conv} we return to
\eqref{eq:vept:final:int:time:bnd}.  Taking expected values we obtain
\begin{align*}
  \E \int_0^{t} \!\! \| \nabla \Vep\|^2 dt
	\leq& \eps (\E \|\Vep(0)\|^2 + C)  +  C \E\left(\sup_{s \in [0,t]} \|\phep(s) \|^2\right),
\end{align*}
where $C = C(\ra, \rab, t, C_0)$ is independent of $\eps > 0$.  Combining this observation
with the previous bound, the proof of Theorem~\ref{prop:conv:vel} is now complete.
\end{proof}

\subsection{Approximation up to Initial Conditions: The Corrector}
\label{eq:corrector}

We next formally derive and then rigorously analyze a refined
approximation of \eqref{eq:B:eqn:vel:mod}--\eqref{eq:B:theta:eps}. By
Theorem~\ref{prop:conv:vel}, the velocity component $M(\theta^0)$ of
the limit system \eqref{eq:vel:zero}--\eqref{eq:theta:zero} well
approximates the velocity field $\bfU^\varepsilon$ of the full system
\eqref{eq:B:eqn:vel:mod}--\eqref{eq:B:theta:eps} in the norm
$L^2([0, t], H^1(\mathcal{D}))$ for each fixed $t > 0$.  Also
Theorem~\ref{Thm:Main:Thm:Precise}, (ii) shows that the invariant
measure of the limit system approximates any invariant state of the
full system, which can be interpreted as a approximation of laws of
solutions as $t \to \infty$.  On the other hand, we do not expect
\eqref{eq:vel:zero}--\eqref{eq:theta:zero} to accurately describe the
behavior of \eqref{eq:B:eqn:vel:mod}--\eqref{eq:B:theta:eps} up to
$t = 0$ due the presence of a (initial time) boundary layer.

We next derive the so called `corrector equation' which provides
effective dynamics for
\eqref{eq:B:eqn:vel:mod}--\eqref{eq:B:theta:eps} and which is globally
well-posed and whose velocity component remains close to the dynamics of
\eqref{eq:B:eqn:vel:mod}--\eqref{eq:B:theta:eps} in
$L^\infty([0, T], L^2(\mathcal{D}))$, that is, even up to time zero.
Note that similar considerations motivate the analysis in
\cite{Wang2004b} which treats such small time approximations in the
deterministic setting.

\subsubsection*{Formal Derivation}

In order to identify multiple time scales in
\eqref{eq:B:eqn:vel:mod}--\eqref{eq:B:theta:eps} we introduce an
additional `slow time' variable $\varsigma = \eps t$.  We then replace
\begin{align*}
  \partial_t \rightarrow \partial_t + \frac{1}{\eps} \partial_\varsigma.
\end{align*}
Under this ansatz the momentum equation \eqref{eq:B:eqn:vel:mod} becomes
\begin{align*}
  \eps(\pd{t} \Uep + \Uep \cdot \nabla \Uep)  + \pd{\varsigma}\Uep - \Delta \Uep
   = \nabla \tilde{p}^\eps  + \ra \hatk \Thep 
\end{align*}
Dropping the terms of order $\varepsilon$ and using Duhamel's formula we obtain
\begin{align}
  \Uep(\varsigma) = e^{-A\varsigma} \Uep(0) + \int_0^{\varsigma} e^{-A(\varsigma -r)} P(\ra \hatk \Thep) dr
  \label{eq:var:frm:ass}
\end{align}
where as in Section~\ref{sec:stokes}, $e^{-A\varsigma}$ denotes
the semigroup whose generator is the Stokes operator $A$.

The form of \eqref{eq:B:eqn:vel:mod} suggests that $\bfU^\eps$
fluctuates rapidly in comparison to $\theta^\eps$.   Under the 
further ansatz that there is a clear separation of time scales between
the motion of $\Uep$ and that of $\Thep$ we suppose
that $\Thep$ is independent of $\varsigma$.  
From \eqref{eq:var:frm:ass} this yields
\begin{align}
    \Uep(\varsigma) 
      &= e^{- A \varsigma} \Uep(0) + A^{-1}(P( \ra \cdot \hatk \Thept)) 
        - e^{- A \varsigma} A^{-1}(P(\ra  \hat{k}\theta^\eps))  \notag\\
      &:= A^{-1}(P( \ra \cdot \hatk \Thept)) + \bfW^\eps(\varsigma) \,,
   \label{eq:my:ass:1}
\end{align}
where $\bfW^\eps$ solves
\begin{align}
  \pd{\varsigma} \bfW^\eps - \Delta \bfW^\eps = \nabla q^\eps, \quad \bfW^\eps(0) = \Uep(0) - \mathbf{y}^\eps
  \quad \text{ and } \; -\Delta \mathbf{y}^\eps = \nabla p^\eps + \ra \hat{k}\theta^\eps(0)
   \label{eq:my:ass:2}
\end{align}
and we have made the further approximation that $\theta^\eps(t) \approx \theta^\eps(0)$ relative to
the slow time scale $\varsigma$.
 
Next,  we return to the original time scale $t$
and obtain the effective dynamics for 
\eqref{eq:B:eqn:vel:mod}--\eqref{eq:B:theta:eps} starting from any initial 
condition $(\theta^\eps_0, \bfU^\eps_0) \in \HH$,
\begin{align}
 &-\Delta \tilde{\bfU}^\eps  = \nabla p^\eps +\ra \cdot \hatk \Thept + \Delta \bfW^{\eps}(t),  
   \quad \nabla \cdot \Uept = 0\,,
  \label{eq:Ulim:eft}\\
 d \Thept &+ \left( \tilde{\bfU}^\eps \cdot \nabla \Thept - \Delta \Thept \right) dt 
            = \rab \cdot \tilde{u}_{d}^{\eps}dt 
           + \sum_{k =1}^N \sigma_k dW^k,  \quad \Thept(0) = \theta^\eps_0,
  \label{eq:Thlim:eft}
\end{align}
where $\bfW^\eps$ solves
\begin{align}
\begin{split}
&\pd{t}\bfW^{\eps} - \frac{1}{\eps}\Delta \bfW^{\eps}   =  \frac{1}{\eps}\nabla q^\eps, \quad \nabla \cdot \bfW^{\eps} = 0,\\
&\bfW^{\eps}(0)= \mathcal{P}_{N^\eps} \Uep_{0} +  \bfY^\eps,
	\quad \textrm{ and } \quad
	-\Delta \bfY^\eps =  \nabla p^\eps +  \ra \cdot \hatk \Thep_{0},\quad \nabla \cdot \bfY^\eps = 0.
\end{split}	
	\label{eq:wep}
\end{align}
We supplement \eqref{eq:Ulim:eft}--\eqref{eq:wep} with boundary
conditions \eqref{eq:bc:homo}.  Note that for technical reasons we
slightly modify the initial condition for $\bfW^\eps$ compared to
\eqref{eq:my:ass:2} by taking
$\tilde{\bfU}^\eps(0) = \mathcal{P}_{N^\eps} \Uep_{0}$, where we
recall that $\mathcal{P}_{N^\eps}$ is the projection onto the first
$N^\eps$ modes of the Stokes operator $A$ as in \eqref{eq:wep:IC:cond}
and $N^\eps$ satisfies
\begin{align}
  \eps (\lambda^{*}_{N^{\eps}})^2 \sim 1.
  \label{eq:wep:IC:cond}
\end{align}
This specification $\tilde{\bfU}^\eps(0)$ is only used to avoid regularity issues at the
initial time in \eqref{eq:Ulim:eft} and as such, a number of other
modifications can be employed.

\subsubsection{Rigorous Error Estimates for the Corrector}

The following theorem asserts that \eqref{eq:Ulim:eft}--\eqref{eq:wep}
approximates \eqref{eq:B:eqn:vel:mod}--\eqref{eq:B:theta:eps} in the
desired norms.

\begin{Thm}
  \label{lem:cor:to:LargePrandlt:og}
  Fix any $\eps > 0$ choose $N_\eps$ satisfying \eqref{eq:wep:IC:cond}. 
  Suppose we are given a sequence $\{\mu^{0,\eps}\}_{\eps \in (0,1)} \subset \Pr(\HH)$ 
  such that
  \begin{align}
  	\sup_{0 < \eps \leq 1} \int_\HH \left(\| \nabla \bfU\|^2 
    +  \exp( \eta \|\bfU\|^2 + \| \theta\|_{L^3}^2)\right) d \mu^{0, \eps}(\bfU, \theta)  
	< \infty.
	\label{eq:uni:IC}
  \end{align}
  For each $\eps > 0$ we consider a martingale solutions
  $(\bfU^\eps, \theta^\eps)$ of
  \eqref{eq:B:eqn:vel:mod}--\eqref{eq:bc:homo} as in
  Proposition~\ref{prop:existence:sol:BEs}, (i).  We suppose that each
  $(\bfU^\eps, \theta^\eps)$ has initial conditions distributed according to
  the distribution $\mu^{0, \eps}$ and satisfies the
  uniform moment bound \eqref{eq:exp:yo:yo:m:c}.  In particular, for
  each $\eps > 0$ the corresponding martingale solution fixes a
  stochastic basis $\mathcal{S}$ and defines
  $(\theta^\eps_0, \bfU^\eps_0) := (\theta^\eps(0), \bfU^\eps(0))$.
  Then,
  \begin{itemize}
  \item[(i)] up to the specification of the stochastic basis $\mathcal{S}$
    and the initial conditions $(\bfU^\eps_0, \theta^\eps_0)$,
    there exists a unique, adapted 
    \begin{align*}
      \Thept \in L^2(\Omega; L^2_{loc}([0,\infty); V) \cap C([0,\infty); H))
    \end{align*}
    solving \eqref{eq:Ulim:eft}--\eqref{eq:wep}.
  \item[(ii)] For any $t > 0$ there is a
    $\gamma_0 = \gamma_0(\rab, \ra, t)$ such that, for any
    $0< \gamma \leq \gamma_0$, $p \geq \gamma$, and $\eps > 0$,
\begin{align}
  \E \left( \sup_{s \in [0,t]} \|\Thept(s) - \theta^{\eps}(s)\|^{p}\right) \leq C\eps^\gamma \,, \quad
  \E\left(  \sup_{s \in [0,t]} \|\Uept(s) - \Uep(s)\|^{p}   \right) \leq C\eps^{\gamma/4} \,,
   \label{eq:cor:to:LargePrandlt:og:est}
\end{align}
where the constants $C = C(\rab, \ra, t, p)$ and $\gamma_0$ and are both independent of $\eps > 0$.
\end{itemize}
\end{Thm}

\begin{proof}
  As in Proposition~\ref{prop:existence:sol:BEs}, (iii), (iv) the well
  posedness of \eqref{eq:Ulim:eft}--\eqref{eq:wep} is standard and can
  be established along similar lines as one would  for the 2D Stochastic
  Navier-Stokes equations.   To see this observe that, although we are working in 3D, we have
  one more degree of smoothing in the constitutive law, \eqref{eq:Ulim:eft}, producing $\tilde{\bfU}^\eps$ from 
  $\tilde{\theta}^\eps$ compared to Biot-Savart in the Navier-Stokes equation.  We omit 
  further details here again referring the reader to e.g. \cite{ZabczykDaPrato1992,FlandoliGatarek1,
  DebusscheGlattHoltzTemam1}.

  To prove \eqref{eq:cor:to:LargePrandlt:og:est} we reuse
  many of the estimates from the proof of Theorem~\ref{prop:conv:vel}.  Taking $\Vept = \Uep - \Uept$ and
  $\phet = \thep -\Thept$, we have
  \begin{align*}
    \pd{t}\phet - \Delta \phet &=
       \rab \cdot \tilde{v}^{\eps}_{d} - \Vept \cdot \nabla \Thept - \Uep \cdot \nabla \phet,
                                 \qquad \phet(0) = 0\,,
  \end{align*}
  and hence repeating the arguments leading to \eqref{eq:phet} we obtain the estimate
  \begin{align}
    \sup_{s \in [0,t ]} \| \phet(s \wedge \tau) \|^2 \leq
        \sup_{s \in [0, t \wedge \tau ]} \left( \| \Thept (s) \|_{L^3}^2 +
                   \rab^2 \right)\int_0^{t \wedge \tau} \|\nabla \Vept(t') \|^2 dt'
   \label{eq:phet:c}
  \end{align}
  for any $t >0$ and any stopping time $\tau \geq 0$. 
  By  \eqref{eq:Ulim:eft} and  \eqref{eq:B:eqn:vel:mod},  $\Vept$ satisfies
   \begin{align}
    \eps(\pd{t} \Uep + \Uep \cdot \nabla \Uep) -
                   \Delta \Vept = \nabla q^\eps + \ra \hatk \phet - \Delta \bfW^{\eps},
    \label{eq:vep:1:c}
  \end{align}
  where $\bfW^{\eps}$ solves \eqref{eq:wep}.
  Referring to \eqref{eq:Ulim:eft} we have
  \begin{align}
	\Uept = A^{-1} (\ra P(\hatk \Thept)) -  \bfW^{\eps} \,,
	\label{eq:Uept:sol}
  \end{align}
  and consequently \eqref{eq:Thlim:eft} and \eqref{eq:wep} yield
  \begin{align}
    d \Uept &= -\pd{t} \bfW^{\eps} \! + \ra A^{-1} P(\hatk d \Thept) \notag\\
            &= -\frac{1}{\eps} (\Delta \bfW^{\eps} + \nabla q^\eps) dt \!
              - \! \ra A^{-1}  P \! \left(\hatk \!\left( \Uept \cdot \nabla\Thept \!
                    - \Delta \Thept - \rab \cdot \tilde{u}_{d}^{\eps}\right) \! \right)  dt
              + \ra \sum_{k =1}^N \! A^{-1} P(\hatk \sigma_k) dW^k\,.
	      \label{eq:vep:2:c}
  \end{align}
  Multiplying \eqref{eq:vep:1:c} by $\eps^{-1}$, subtracting
   \eqref{eq:vep:2:c} and rearranging we obtain
  \begin{align}
    d \Vept - \frac{1}{\eps}\Delta \Vept dt =&
               \frac{1}{\eps}\left(\nabla \tilde{q}^\eps + \ra \hatk \phet \right) dt
             + \!  \left( \ra A^{-1} P \! \left( \hatk \! \left( \Uept \cdot \nabla \Thept
               - \Delta \Thept - \rab \cdot \tilde{u}_{d}^{\eps}\right) \! \right) \!
               - \Uep \cdot \nabla \Uep \right) \! dt
               \notag\\
             &- \ra \sum_{k =1}^N \! A^{-1} P( \hatk \sigma_k) dW^k,
            \qquad  \Vept(0) = (I- \mathcal{P}_N) \Uep_0, 
  \label{eq:vept:evo:eq:0:c}
  \end{align}
  with $\nabla \cdot \Vept = 0$.  Here note the close similarity between \eqref{eq:vept:evolution:eq:0} 
  and \eqref{eq:vept:evo:eq:0:c}; in view of \eqref{eq:Uept:sol}, \eqref{eq:wep} the primary distinction here
  is in the initial condition.   
  
  As above in \eqref{eq:vept:evolution:eq}, the It\=o formula implies
  \begin{align}
  d\|\Vept\|^2 + \frac{2}{\eps} \|\nabla \Vept\|^2 dt =&
       \frac{2}{\eps} \ra \langle \phet, \tilde{v}^{\eps}_d \rangle dt
      + 2 \left\langle \ra A^{-1} \! P \! \left( \hatk \left( \Uept \cdot\nabla\Thept
          - \Delta \Thept - \rab \cdot \tilde{u}^{\eps}_{d}\right)\right) 
          - \Uep \cdot \nabla \Uep, \Vept \right\rangle dt
           \notag\\
  &+ \ra^2 \sum_{k =1}^N | A^{-1} P( \hatk \sigma_k)|^2 dt -
    2 \ra \sum_{k =1}^N \langle A^{-1} P(\hatk \sigma_k), \Vept \rangle dW^k 
    \notag\\
    :=& (T_1 + T_2 + T_3 + T_4 + T_5 + T_6) dt + S dW.
 \label{eq:vept:evo:eq:c}
  \end{align}
  We now estimate \eqref{eq:vept:evo:eq:c} with bounds similar to
  \eqref{eq:T1:bbd:og}--\eqref{eq:T5:bbd:og}.   Here, bounds on 
  $\bfU^0$ need to be replaced with appropriate estimates for $\Uept$.
  For the terms $ T_2, T_4$ we simply treat $\Uept$ terms as
  \begin{align}
    \|\Uept(t)\|^2  \leq C (\ra^2 \|\Thept(t)\|^2 + \|\bfW(t)\|^2) \leq  
    C (\ra^2 \|\Thept(t)\|^2 + \| \bfU^\eps_0\|^2 + \ra^2 \|\theta^\eps_0\|^2) \,.
        \label{eq:d:cor:u:est:2}
  \end{align}
  The estimate \eqref{eq:T5:bbd:og} on the term $T_5$ involves an
  $L^\infty$ bound on $\Uept$ and thus requires a bit more care. In this case
  \begin{align*}
    \|\Uept(t)\|_{L^\infty}^2 \leq C \|\Uept(t)\|_{H^2}^2 \leq C (\ra^2 \|\Thept(t)\|^2 + \|\bfW(t)\|^2_{H^2}).
  \end{align*}
  By standard properties of analytic semigroups, the inverse Poincar\' e inequality, and the form of $\bfW(0)$  one has 
  \begin{align}
  \|\bfW(t)\|^2_{H^2}  
  &= \|e^{At/\eps}\bfW(0)\|^2_{H^2}  \leq  C \|\bfW(0)\|^2_{H^2} 
    \leq 
    C\ra^2  \|\theta^\eps_0\|^2 + C \| \mathcal{P}_{N^\eps}\Uept_0\|^2_{H^2}
    \notag\\
  &\leq
    C\ra^2  \|\theta^\eps_0\|^2 + C (\lambda_{N^\eps}^*)^2 \|\Uept_0\|^2 \,,
  \end{align}  
  where $C$ is independent of $\eps \in (0, 1)$.  Combining these two estimate and again taking advantage of 
  the cancelation from $\Uep = \Vept + \Uept$
  \begin{align}
  	|T_5| 
	\leq  \frac{1}{4 \eps} \|\nabla \Vept\|^2 
               + 4 \eps \| \Uept\|_{L^\infty}^2 \|\Uep\|^2
         \leq	\frac{1}{4 \eps} \|\nabla \Vept\|^2 + C(\ra^2 + 1) ( \|\theta^\eps_0\|^4 + \|\Uep_0\|^4 + \|\Uep\|^4).
                 \label{eq:d:cor:u:est:1}
  \end{align}
  Observe that in comparison to \eqref{eq:T5:bbd:og}, the additional power of $\eps$
  is used to cancel $(\lambda_{N^\eps}^*)^2$. 
  
  Combining the analogues of \eqref{eq:T1:bbd:og}--\eqref{eq:T34:bbd:og} with \eqref{eq:d:cor:u:est:2} and 
  using \eqref{eq:d:cor:u:est:2} with \eqref{eq:vept:evo:eq:c} we obtain
  \begin{align}
    d\|\Vept\|^2 \!  + \! \frac{1}{\eps}\|\nabla \Vept\|^2 dt
    &\leq \frac{C \ra^2}{\eps} \|\phet \|^2 dt 
      + \!C(1\!+ \!\ra^4)(1 \!+\! \rab^2) (\|\Thept \|^4 \! + \| \theta^\eps_{0}\|^4\!
      + \|\bfU^\eps_{0}\|^4 \! + \|\Uep\|^4 \! + 1)dt
      \notag\\
    &-2\ra \sum_{k =1}^N \langle A^{-1}( P(\hatk \sigma_k)), \Vept \rangle dW^k \,,
      \label{eq:vept:evo:est:c}
  \end{align}
  where the constant $C> 0$ is independent of $\ra, \rab$, and
  $\eps >0$.  We now use \eqref{eq:vept:evo:est:c} with
  \eqref{eq:phet:c} and repeate the stopping time argument as in
  \eqref{eq:phet:big:man}--\eqref{eq:re:sum:trick} to infer the first
  part of the \eqref{eq:cor:to:LargePrandlt:og:est}.\footnote{Note
    that the loss of the $\eps$ in front of the second term after the
    inequality in \eqref{eq:vept:evo:est:c} in comparison
    \eqref{eq:vept:evo:est} does not charge the ultimate outcome of
    this bound as we only required an $\eps$-independent upper bound
    for $\mathcal{M}_\eps$ in \eqref{eq:phet:big:man:2}.}

  We turn next to the the convergence of the velocity fields, 
  the second part of  \eqref{eq:cor:to:LargePrandlt:og:est}. We obtain from
  \eqref{eq:vept:evo:eq:c} and the pointwise bounds yielding
  the drift terms in \eqref{eq:vept:evo:est:c} that
  \begin{align}
    \|\Vept(t)\|^2
    &\leq  \exp\left( - \frac{t}{\eps}\right) \|\Vept_0\|^2 +
        C\int_0^t \exp\left( - \frac{t-s}{\eps} \right)
          \left(\frac{\ra^2}{\eps}\|\phet(s) \|^2
          +  \mathcal{R}_\eps(s) \right)ds
          + \mathcal{X}_\eps(t),
    \notag\\
    &\leq \exp\left( - \frac{t}{\eps}\right) \|\Vept_0\|^2 
          +  C \ra^2 \sup_{s \in [0,t]} \| \phet(s) \|^2
          +\eps \sup_{s \in [0,t]}\mathcal{R}_\eps(s)
          +\mathcal{X}_\eps(t) \,,
          \label{eq:key:cor:est:1}
  \end{align}
  where
  \begin{align*}
    \mathcal{R}_\eps(t) := C(1+ \ra^4)(1 + \rab^2)
    (\|\Thept \|^4  + \|\theta^\eps_{0}\|^4
      + \|\bfU^\eps_{0}\|^4 + \|\Uep\|^4  + 1)
  \end{align*}
  and
  \begin{align*}
    \mathcal{X}_\eps(t) := - 2\ra\int_0^t\exp\left( - \frac{t-s}{\eps} \right)
    \sum_{k =1}^N \langle A^{-1}( P(\hatk \sigma_k)), \Vept \rangle dW^k
    =: \int_0^t \exp \left( - \frac{t-s}{\eps} \right) g(s) dW.
  \end{align*}
Using the inverse Poincar\' e inequality and \eqref{eq:wep:IC:cond} one has
\begin{align}
  \exp\left( - \frac{t}{\eps}\right) \|\Vept_0\|^2 \leq  \|(I - P_{N^\eps}) \Uep_0\|^2 
  \leq C (\lambda_{N^\eps}^*)^{-1}  \|\nabla \Uep_0\|^2
  \leq C \eps  \|\nabla \Uep_0\|^2 \,.
    \label{ices}
\end{align}
Therefore combining \eqref{eq:key:cor:est:1} with \eqref{ices}, using the bound already obtained
for $\|\phet(s) \|$ in (\ref{eq:cor:to:LargePrandlt:og:est}) and the uniform bounds (\ref{eq:exp:yo:yo:m:c}), 
(\ref{eq:uni:IC})
\begin{align}
  \E \sup_{s \in [0, t]}   \|\Vept(t)\|^p 
    \leq C(\eps^{p/2} + \eps^\gamma) + C\E \sup_{s \in [0,t]} |\mathcal{X}_\eps(s)|^{p/2}
          \label{eq:key:cor:est:3}
\end{align}
for any $p > 0$, where $\gamma = \min\{p, \gamma_0\}$ is obtained from the bound $\|\phet(s) \|$ and 
the constant $C = C(\ra, \rab, t, p)$ is independent of $\eps > 0$.

In order to estimate $\mathcal{X}_\eps$  observe that this process satisfies
\begin{align*}
    d \mathcal{X}_\eps + \frac{1}{\eps}\mathcal{X}_\eps dt = g dW, \quad \mathcal{X}_\eps(0) = 0,
\end{align*}
and hence, by the It\=o lemma,
\begin{align*}
	d \mathcal{X}_\eps^2 + \frac{2}{\eps} \mathcal{X}_\eps^2 dt = g^2 dt +  2 g \mathcal{X}_\eps dW \,.
\end{align*}  
Consequently,
\begin{align}
    \E \sup_{s \in [0,t]} |\mathcal{X}_\eps(s)|^{p/2}
       \leq C\E \sup_{s \in [0,t]} \left| \int_0^s g \mathcal{X}_\eps dW\right|^{p/4} + C\E \left( \int_0^t g^2 ds\right)^{p/4}.
          \label{eq:key:cor:est:4}
\end{align}
With the Burkholder-David-Gundy inequality and Young's inequality we have
\begin{align}
  \E \sup_{s \in [0,t]} \left| \int_0^s g \mathcal{X}_\eps dW\right|^{p/4}
  \leq
      C \E \left( \int_0^t g^2 \mathcal{X}_\eps^2 ds \right)^{p/8}
  \leq \frac{1}{2} \E  \sup_{s \in [0,t]} |\mathcal{X}_\eps|^{p/2} + C \E \left( \int_0^t \| \Vept\|^2 ds \right)^{p/4}.
          \label{eq:key:cor:est:5}
\end{align}
With \eqref{eq:key:cor:est:4}, \eqref{eq:key:cor:est:5},  the bound \eqref{eq:key:cor:est:3} now yields
\begin{align}
  \E \sup_{s \in [0, t]}   \|\Vept(t)\|^p 
    \leq C(\eps^{p/2} + \eps^\gamma) + C \E \left( \int_0^t \| \Vept\|^2 ds \right)^{p/4}.
          \label{eq:key:cor:est:3}
\end{align}
On the other hand, from (\ref{eq:vept:evo:est:c}), using for a second time
the existing bounds on $\|\phet(s) \|$ in (\ref{eq:cor:to:LargePrandlt:og:est}) and (\ref{eq:exp:yo:yo:m:c}),
we have
\begin{align}
    \E \int_0^t\|\nabla \Vept\|^2 dt
    &\leq \eps\E \|\Vept_0\|^2 + C \ra^2 \int_0^t \E\|\phet \|^2 ds 
      + \eps \E \int_0^t \mathcal{R}_\eps ds
    \leq C(\eps+ \eps^\gamma),
   \label{eq:key:cor:est:6}
\end{align}
where the constant $C$ depends on $t$, $\ra$, $\rab$ but again is
independent of $\eps > 0$.  When $p \leq 4$ the second portion of the
desired inequality (\ref{eq:cor:to:LargePrandlt:og:est}) now follows
from \eqref{eq:key:cor:est:3}, \eqref{eq:key:cor:est:6}, and H\"older's
inequality.  On the other hand, when $p > 4$ then we estimate
\begin{align*}
  \E \left( \int_0^t \| \Vept\|^2 ds \right)^{p/4} 
  \leq 
  \E \sup_{s \in [0,t]} \|\Vept\|^{p-2} + \E \int_0^t \|\Vept\|^2 ds
  \leq \sup_{s \in [0,t]} \|\Vept\|^{p-2} + C(\eps+ \eps^\gamma),
\end{align*}
so that the second part of (\ref{eq:cor:to:LargePrandlt:og:est})
follows in this later case with \eqref{eq:key:cor:est:6} and an
iterative argument.  This completes the proof of
Theorem~\ref{lem:cor:to:LargePrandlt:og}, (ii).
\end{proof}

\section{Convergence of the Nusselt number}
\label{sec:conv:nus}

In this final section we prove Corollary \ref{cor:main}, note that
this is not an immediate consequence of
Theorem~\ref{Thm:Main:Thm:Precise}.  Indeed, the Nusselt number $Nu$
(cf. \eqref{eq:nu:unique}) is defined as a statistical average of the
observable
\begin{align}\label{dfph}
  \phi_{Nu}(\textbf{u},\theta)
  =\frac{1}{|\mathcal{D}|}\int_{\mathcal{D}}u_2 \theta dx,
\end{align}
but (\ref{eq:ext:obs}) is not satisfied since  
 $\phi_{Nu}\notin V(\HH)$. Nevertheless,  Corollary \ref{cor:main} follows from a
combination of Theorem~\ref{Thm:Main:Thm:Precise} with exponential moment
bounds on the unique invariant measures.

\begin{proof}[Proof of Corollary \ref{cor:main}]
  In two space dimensions, with $N=\infty$ and $Pr=Pr(\ra,\rab)$ large
  enough (or $\eps = Pr^{-1}$ small enough), \cite[Theorem
  1.3]{FoldesGlattHoltzRichardsWhitehead2015} and Theorem
  \ref{Thm:Main:Thm:Precise} yields respectively  unique invariant
  measures $\mu_{\eps}$ of
  \eqref{eq:B:eqn:vel:mod}--\eqref{eq:B:theta:eps} and $\mu_0$ for
  \eqref{eq:vel:zero}-\eqref{eq:theta:zero}.  We aim to show that
  $\lim_{\eps \rightarrow 0}(Nu)_{\eps} = (Nu)_{0}$. To simplify the
  notation, write $\phi$ instead of $\phi_{Nu}$ defined by
  \eqref{dfph}.

   Consider
  a smooth cut-off function $\psi:[0,\infty)\rightarrow [0,1]$
  satisfying $\psi \equiv 1$ for $x\in[0,1]$, $\psi \equiv 0$ for
  $x\geq 2$, and $|\psi '|\leq 2$ everywhere.  Denote
  $\psi_{R}(\cdot) := \psi(\cdot/R)$ and write
\begin{align*}
&\left|\int_{\HH} \phi(\textbf{u},\theta)
     d\mu_{\eps}(\textbf{u},\theta) - \int_{H}\phi (L(\theta))
  d\mu_{0}(\theta)\right|  \\
&\leq
   \left|\int_{\HH} \phi(\textbf{u},\theta)\psi_{R}(\|\mathbf{u}\|^2 +
     \|\theta\|^2)d\mu_{\eps}(\textbf{u},\theta) 
         - \int_{H}\phi (L(\theta))\psi_{R}(\|L(\theta)\|^2)
      d\mu_{0}(\theta)\right| \\
& \quad + \left|\int_{\HH} \phi (\textbf{u},\theta)
  (1-\psi_{R}(\|\mathbf{u}\|^2 + \|\theta\|^2))
  d\mu_{\eps}(\textbf{u},\theta)\right| + \left|\int_{H}\phi(L(\theta))
   (1-\psi_{R}(\|L(\theta)\|^2))d\mu_{0}(\theta)\right| \\
   & =: I_{1} + I_{2} + I_{3}.
\end{align*}
We bound the first term by observing that
$\phi_{R}(\textbf{u},\theta) :=
\phi(\textbf{u},\theta)\psi_{R}(\|\mathbf{u}\|^2 + \|\theta\|^2)$
satisfies
\begin{align}\label{eq:semi:bound}
\left[\phi_{R}\right]_{\eta} \leq C(R+1),
\end{align}
for a constant $C>0$.  It follows from \eqref{eq:semi:bound} and
Theorem \ref{Thm:Main:Thm:Precise} that
$I_1\leq \tilde{C}(R+1)\eps^{\tilde{q}}$.  We control $I_2$ using
Markov's inequality and \eqref{eq:exp:yomoment:cond}, which yields,
for any $\eta>0$ sufficiently small
\begin{align*}
\mu_{\eps}(E_{R}^{c}) &\leq e^{-\eta R}\int_{\HH}\exp(\eta(\|\mathbf{u}\|^{2} + \|\theta\|^{2}))
d\mu_{\eps}(\mathbf{u},\theta)
\leq C_{0} e^{-\eta R}.
\end{align*}
uniformly in $\eps>0$, where
$E_{R}= \{\|\mathbf{u}\|^{2} + \|\theta\|^2 \leq R\}$.  Notice that
$1-\psi_{R}(\|u\|^2 + \|\theta\|^2) \leq \mathbbm{1}_{E_{R}^{c}}$, and
observe that for fixed $\eta>0$, on the set $E_{R}^{c}$ we have
$|\phi_{Nu}(\mathbf{u},\theta)|\leq \exp(\eta(\|\mathbf{u}\|^{2} +
\|\theta\|^{2})/2)$
for $R=R(\eta)$ sufficiently large.  By H\"{o}lder's inequality we
obtain
\begin{align*}
I_{2} &\leq  C\left(\int_{\HH}\exp(\eta(\|\mathbf{u}\|^{2} + \|\theta\|^{2}))d\mu_{\eps}(\mathbf{u},\theta)
\right)^{1/2}\left(\mu_{\eps}(E_{R}^{c})\right)^{1/2} \leq C C_{0}^{1/2} e^{-\eta R/2}.
\end{align*}
We can control $I_3$ similarly by using exponential moment bounds for the invariant measure $\mu_0$,
and the estimate $\|L(\theta)\|\leq C\|\theta\|$.
Thus, 
\begin{align}
\left|\int_{\HH} \phi_{Nu(\eps)}(\textbf{u},\theta)d\mu_{\eps}(\textbf{u},\theta) - \int_{H}\phi_{Nu(0)}(L(\theta))d\mu_{0}(\theta)\right|  \leq C((R+1)\eps^{\tilde{q}} + e^{-\eta R/2}),
\end{align}
and we can make this expression less than any $\delta>0$ by taking $R=R(\delta)$ sufficiently large
and then $\eps = \eps(R,\delta)$ sufficiently small.
\end{proof}

\begin{Rmk}
  From the proof above we can easily infer that Corollary
  \ref{cor:main} will also apply to any observable $\phi$ on the
  extended phase space that is locally Lipschitz and sub-exponential
  at infinity (for a sufficiently small exponential power $\eta>0$
  dictated by \eqref{eq:exp:yomoment:cond}).  We have chosen the
  convergence of the Nusselt number to emphasize the physical
  significance of our results.
\end{Rmk}

\appendix

\section{Appendix: Moment Bounds for Stochastic Drift-Diffusion Equations}
\label{sec:mom:est:drift:diff:eqn}

In this appendix we collect some moment bounds proved in
\cite{FoldesGlattHoltzRichardsWhitehead2015} which have been used in
the analysis above.

As in \cite{FoldesGlattHoltzRichardsWhitehead2015} we consider the
following class of stochastic divergence-free drift diffusion systems
\begin{align}
  d \xi + \bfV \cdot \nabla \xi dt = (\rab \cdot v_3 + \Delta \xi)dt + \sum_{k =1}^N \sigma_k dW^k, \quad \xi(0) = \xi_0
  \label{eq:bous:drift:diff:stoch:sys:1}
\end{align}
evolving on the three dimensional domain
$\DD = [ 0, L]^2 \times [0,1]$.  Here $\rab >0$ is a fixed parameter
and $\bfV = (v_1, v_2, v_3)$ is any sufficiently regular and adapted,
divergence free vector field.  Both $\bfV$ and $\xi$ are supposed to
satisfy the boundary condition \eqref{eq:bc:homo}.  Recall that by the
change of variable $T = \xi + \rab(1 - z)$ we may reformulate
\eqref{eq:bous:drift:diff:stoch:sys:1} as
\begin{align}
  d T + \bfV \cdot \nabla T dt = \Delta Tdt + \sum_{k =1}^N \sigma_k dW^k, \quad T(0) = T_0  = \xi_0 + \rab(1 - z)\,,
    \label{eq:bous:drift:diff:stoch:sys:2}
\end{align}
where $\bfV$ and $T$ satisfy boundary conditions  \eqref{eq:bc}.
As such, bounds for $\xi$ solving \eqref{eq:bous:drift:diff:stoch:sys:1} immediately translate to bounds for $T$.

In  \cite{FoldesGlattHoltzRichardsWhitehead2015} we prove:
\begin{Prop}
\label{prop:lets:make:the:most:of:the:moments:that:count}
Suppose that $\bfV \in L^2_{loc}([0,\infty); V_1 \cap (H^2(\DD))^3) \cap C([0,\infty); H_1)$ a.s. and is $\mathcal{F}_t$-adapted.
Fix any $p \geq 2$ and any initial condition $\xi_0 \in H \cap L^p(\DD)$ which is $\mathcal{F}_0$ measurable with
\begin{align*}
   \E \exp( \eta \| \xi_0 \|_{L^p}^2) < \infty,
\end{align*}
for some $\eta > 0$.
Then there exists $\eta_0 = \eta_0(\sigma, \rab, p) > 0$ such that for any $t \geq 0$ and any positive $\eta \leq \eta_0$,
\begin{align}
    \E \exp \left(\frac{\eta}{2^{p/2+2}} \sup_{s \in [0, t]} \|\xi\|_{L^p}^2 \right) \leq C_1 \E \exp \left(  \eta  \|\xi_0\|_{L^p}^{2} + \eta p t (\|\sigma\|_{L^p}^2  + 2^{p/2}( 4\rab^2 + 1))\right)    \label{eq:gen:drift:diff:bnd:1}
\end{align}
for a constant $C = C( \rab,p)$ independent of $t$, $\eta$, $\xi_0$,
and $\bfV$.
Furthermore,
\begin{align}
  \E \exp \left( \frac{\eta}{2^{p/2+2}} \|\xi(t)\|^2_{L^p}  \right)
  \leq C
  \E \exp \left( \eta (e^{- \kappa t} \|\xi_0\|^2)   \right) \,,
      \label{eq:gen:drift:diff:bnd:2}
\end{align}
where again $C = C( \rab,p, \|\sigma\|_{L^p}, \DD)$ and $\kappa = \kappa(\rab, \DD) > 0$ are independent of $t$, $\eta$, $\xi_0$, and $\bfV$.
\end{Prop}

We now return to the infinite Prandtl system
\eqref{eq:vel:zero}--\eqref{eq:theta:zero} and recall a bound
analogous to \eqref{eq:gen:drift:diff:bnd:2} but which uses more of
the specific structure of the velocity equation.
\begin{Prop}\label{prop:inf:prnd:bound} Fix an initial condition
  $\Tho_0 \in H$ which is $\mathcal{F}_0$ measurable, and let
  $\Tho=\Tho(t,\Tho_0)$ denote the corresponding solution to
  \eqref{eq:vel:zero}--\eqref{eq:theta:zero}.  There is a universal
  constant $\eta^{*}>0$ such that for any $t>0$ and
  $\eta \in (0,\eta^*]$, there exists $C=C(\ra,\rab)>0$ such that
\begin{align*}
\E\left(\exp\left(\eta\|\Tho\|^{2} 
    + \frac{\eta e^{-t/4}}{4}\int_{0}^{t}\|\nabla\Tho\|^{2}ds\right)\right)
    \leq C \exp\left(\eta (1+4\ra\rab)e^{-t/2}\|\Tho_{0}\|^{2}\right).
\end{align*}
\end{Prop}
\noindent
The proof of Proposition \ref{prop:inf:prnd:bound} can be found in \cite{FoldesGlattHoltzRichardsWhitehead2015}.

\section{Gradient Estimates on the Markov Semigroup}
\label{sec:grad:est:markov}

In this section we establish the gradient bound for the Markov semigroup generated by \eqref{eq:vel:zero}--\eqref{eq:theta:zero} in
order to prove  \eqref{eq:ASF:type:bnd}.
For this purpose we begin by briefly recalling how \eqref{eq:ASF:type:bnd} is translated to a control problem
through the use of Malliavin calculus.  We refer to e.g. \cite{Nualart2009} or \cite{NourdinPeccati2012} for further general background on this subject
and to \cite{HairerMattingly06, HairerMattingly2011, FoldesGlattHoltzRichardsThomann2013} for the application of this formalism
in a setting close to ours.

Define the random operators
\begin{align}
   \JJ_{0,t} \xi := \lim_{\epsilon \to 0}  \frac{\theta^0(t, \theta_0 + \epsilon \xi,W) - \theta^0(t, \theta_0,W)}{\epsilon}
   \label{eq:J:op}
\end{align}
for any $\xi \in H$ and
\begin{align}
   \AAA_{0,t} w :=  \lim_{\epsilon \to 0}  \frac{\theta^0(t, \theta_0, W + \epsilon \int_0^\cdot w) - \theta^0(t, \theta_0, W)}{\epsilon}
   \label{eq:A:op}
\end{align}
for any $w \in L^2(\Omega; L^2([0,t]; \RR^N))$.    Here $\AAA_{0,t} w = \langle \MD \theta^0, w\rangle$, where the unbounded operator
$\MD: L^2(\Omega; H) \mapsto L^2(\Omega; L^2(0,t, \RR^N) \otimes H)$ is the Malliavin derivative and $w$
is any element in the domain of the dual operator $\delta$ of $\MD$.

For our purposes it is sufficient to recall that any $\mathcal{F}_t$-adapted process in $ L^2(\Omega; L^2([0,t]; \RR^N))$ belongs to the domain of $\delta$ and $\delta(w)$
corresponds to the It\={o} integral of $w$ so that
\begin{align}
  \E \langle \MD X, w\rangle = \E \left( X \int_0^t w dW \right)
  \label{eq:mal:IBP}
\end{align}
for any $X \in \mbox{Dom}(\MD)$ and any $\mathcal{F}_t$-adapted $w$.  This is a special case of the \emph{Malliavin integration by parts formula}.
We furthermore recall that $\MD$ satisfies a chain rule namely  that if $\phi \in C^1(H)$ and $\theta \in \mbox{Dom}(\MD)$ then $\phi(\theta) \in \mbox{Dom}(\MD)$ and
\begin{align}
  \MD \phi(\theta)  = \nabla \phi(\theta) \MD \theta.
  \label{eq:mal:chain}
\end{align}
Combining \eqref{eq:mal:IBP}--\eqref{eq:mal:chain} and making use of the It\={o} isometry we infer that,
\begin{align}
  \nabla P_t^0 \phi(\Trs_0)  \xi
    =& \E \left( \nabla \phi( \theta^0(t, \Trs_0)) \JJ_{0,t} \xi \right)
    = \E \left( \phi( \theta^0(t, \Trs_0)) \int_0^t w dW  \right)
    +\E \left( \nabla \phi( \theta^0(t, \Trs_0)) \left( \JJ_{0,t}  \xi - \AAA_{0,t}w \right) \right)
    \notag\\
    \leq& \sqrt{ P_t^0 (|\phi(\Trs)|^2)} \left(\E \int_0^t |w|^2 dt \right)^{1/2} + \sqrt{ P_t^0 (\|\nabla\phi(\Trs)\|^2)} \left( \E \| \JJ_{0,t}  \xi - \AAA_{0,t}w\|^2 \right)^{1/2}
    \label{eq:fun:cmp}
\end{align}
for any $\phi \in C^1_b(H)$, $\Trs_0 \in H$ and any (adapted) $w \in L^2(\Omega; L^2([0,t]; \RR^N))$.

Our desired bound \eqref{eq:ASF:type:bnd} follows from \eqref{eq:fun:cmp}  if, for every $\xi \in H$ with $\| \xi\| =1$ there is (adapted)
$w = w(\xi) \in L^2([0,\infty); \RR^N)$ such that
\begin{align}
   \E \| \JJ_{0,t}  \xi - \AAA_{0,t}w(\xi)\|^2 &\leq C \exp(2\eta \| \Trs_0\|^2) \delta(t)\,,
   \label{eq:suff:cond:ASF:bnd:2} \\
   \sup_{\|\xi\| = 1} \E \int_0^\infty |w(\xi)|^2 dt &\leq C \exp(2\eta \| \Trs_0\|^2) \,,
   \label{eq:suff:cond:ASF:bnd:1}
\end{align}
where $\delta(t) \to 0$ as $t \to \infty$ and $C$, $\eta$, and $\delta$ are independent of $\theta_0$.

To solve the control problem \eqref{eq:suff:cond:ASF:bnd:2}--\eqref{eq:suff:cond:ASF:bnd:1} we observe that \eqref{eq:J:op} and \eqref{eq:A:op} admit explicit characterizations
as linearizations of \eqref{eq:vel:zero}--\eqref{eq:theta:zero}.
For any $\xi \in H^0$ we let $\rho(t) = \rho(t,\xi) := \JJ_{0,t} \xi$, which satisfies
\begin{align}
 &\partial_t \rho + \Uo \cdot \nabla \rho +  \bfV^0 \cdot \nabla \theta^0  = \rab \cdot v^{0}_d  + \Delta \rho, \quad
 - \Delta \bfV^0 = \nabla p + \Ra \hatk \rho,
 \quad
 \nabla \cdot \bfV^0 = 0\,,
 \quad
 \rho(0) = \xi,
\label{eq:grad:system:0}
\end{align}
supplemented by boundary conditions as in
\eqref{eq:bc:homo}.\footnote{ Notice that \eqref{eq:grad:system:0} can
  also be written as
\begin{align*}
 &\partial_t \rho + (L \theta^0) \cdot \nabla \rho +  (L \rho) \cdot \nabla \theta^0  = \rab (L \rho)  + \Delta \rho,
 \quad
 \rho(0) = \xi,
\end{align*}
where $L = \ra A^{-1} P \hat{k}$ and $A$ is the Stokes operator, $P$;
cf. \eqref{eq:stokes:steady} and \eqref{eq:theta:zero:Juraj:form}
above.  Similar formulations can also be given for
\eqref{eq:grad:system:1}, \eqref{eq:grad:system:2}.} On the other
hand, setting $\tilde{\rho} := \AAA_{0,t}w$ for any
$w \in L^2([0,t], \RR^N)$ we have
\begin{align}
 &\partial_t \tilde{\rho} + \Uo \cdot \nabla \tilde{\rho} +  \tilde{\bfV}^0 \cdot \nabla \theta^0  = \rab \cdot v^{0}_d  + \Delta \tilde{\rho} + \sum_{k =1}^N \sigma_k w_k,
 \;
 -\Delta \tilde{\bfV}^0 = \nabla p + \Ra \hatk \tilde{\rho},
 \; \nabla \cdot \tilde{\bfV}^0 = 0\,, \;
 \tilde{\rho}(0) = 0,
\label{eq:grad:system:1}
\end{align}
again with boundary conditions as in \eqref{eq:bc:homo}.

Denote $\bar{\rho}(t) =  \bar{\rho}(t,\xi, w) = \rho - \tilde{\rho}$
and $\bar{\bfV} :=  \bfV - \tilde{\bfV}$
for any $w \in L^2([0,\infty);\RR^N)$ and $\xi \in H$.
We now choose  $w$ as a function of $\xi$  as follows.
 Let $P_N$ be the projection on the first $N$ eigenfunctions of the Laplacian  with boundary
conditions as in \eqref{eq:bc:homo}.
Set $w(t) := \sigma^{-1} \lambda P_N \bar{\rho}$, where
$\lambda > 0$ and $N$ will be selected below.\footnote{Of course the choice of $N$
will determine the number of modes subject to stochastic perturbation. Observe that $w$ is well defined as $\{\sigma_k\}_{k = 1}^N$ is the set
of the first  $N$ (nonzero) eigenvectors of the Laplacian. }
Relative to this choice of $w = w(\xi)$, $\bar{\rho}$ satisfies
\begin{align}
 &\partial_t \bar{\rho} + \Uo \cdot \nabla \bar{\rho} +  \bar{\bfV}^0 \cdot \nabla \theta^0  = \rab \cdot \bar{v}^{0}_d  + \Delta \bar{\rho} -  \lambda P_N \bar{\rho},
 \quad
 \Delta \bar{\bfV}^0 = \nabla p + \Ra \hatk \bar{\rho},
 \quad
 \nabla \cdot \bar{\bfV}^0 = 0,
 \quad
 \bar{\rho}(0) = \xi.
\label{eq:grad:system:2}
\end{align}
Testing  \eqref{eq:grad:system:2} with $\bar{\rho}$ and $\bar{\bfV}^0$ respectively, and using that both $\Uo$ and $\bar{\bfV}^0$ are divergence free vector fields,
 we obtain
\begin{align}
  \frac{d}{dt} \|\bar{\rho}\|^2 + 2\|\nabla \bar{\rho}\|^2 + 2\lambda \| P_N \bar{\rho}\|^2 = 2\int_\DD \left(\rab \bar{v}^{0}_d -\bar{\bfV}^0 \cdot \nabla \theta^0\right) \bar{\rho} dx
  \label{eq:sim:ASF:est:1}
\end{align}
and
\begin{align}\label{eq:est:vbz}
\|\nabla \bar{\bfV}^0\| \leq \Ra \|\bar{\rho}\|.
\end{align}
With standard Sobolev embeddings and \eqref{eq:est:vbz} we have, for any $\eta>0$,
\begin{align}
  \left| \int_\DD  \left(\rab  \bar{v}^{0}_d -\bar{\bfV}^0 \cdot \nabla \theta^0\right) \bar{\rho} dx  \right|
  &\leq \| \bar{\bfV}^0 \|_{L^6} \| \nabla \theta^0 \| \| \bar{\rho} \|_{L^3} + \rab \| \bar{\bfV}^0 \| \| \bar{\rho} \| \notag\\
  &\leq \tilde{C}\| \nabla \bar{\bfV}^0 \| \| \nabla \theta^0 \| \| \bar{\rho} \|^{1/2} \|\nabla \bar{\rho}\|^{1/2} + \rab\| \nabla \bar{\bfV}^0 \| \|\bar{\rho}\|
  \notag\\
    &\leq \tilde{C}\ra \| \nabla \theta^0 \| \| \bar{\rho} \|^{3/2} \|\nabla \bar{\rho}\|^{1/2} + \tilde{C} \ra\rab \|\bar{\rho}\|^{2}
  \notag\\
  &\leq   \|\nabla \bar{\rho}\|^{2} + (\tilde{C} (\ra)^{4/3} \| \nabla \theta^0 \|^{4/3} + \tilde{C}\ra\rab) \| \bar{\rho} \|^{2} \notag \\
  &\leq \|\nabla \bar{\rho}\|^{2} +  (\eta \| \nabla \theta^0 \|^{2} + C) \| \bar{\rho} \|^{2} \,,
    \label{eq:sim:ASF:est:2}
\end{align}
where $C = C(\Ra, \rab, \eta)=\frac{\tilde{C}\ra^{4}}{\eta^{2}} + \tilde{C} \ra\rab$ and $\tilde{C}$ is a universal
constant.
Also since $P_N$ and $-\Delta$ commute we have for $Q_N := I - P_N$
\begin{equation}\label{eq:lbd:rbr}
\|\nabla \bar{\rho}\|^2 = - \langle P_N \bar{\rho}, \Delta P_N \bar{\rho} \rangle - \langle Q_N \bar{\rho}, \Delta Q_N \bar{\rho} \rangle
= \|\nabla P_N\bar{\rho}\|^2 + \|\nabla Q_N\bar{\rho}\|^2 \geq
\|\nabla Q_N\bar{\rho}\|^2 \geq \lambda_N \|Q_N\bar{\rho}\|^2 \,,
\end{equation}
where the last inequality follows from the generalized Poincar\' e inequality.
Choose $2\lambda =  \lambda_N$ (with $N$ to be chosen below)
and combine \eqref{eq:sim:ASF:est:1} and \eqref{eq:sim:ASF:est:2} to infer
\begin{align*}
  \frac{d}{dt} \|\bar{\rho}\|^2 + ( \lambda_N - (\eta_0 \| \nabla \theta^0 \|^{2} + C))  \| \bar{\rho}\|^2 \leq 0,
\end{align*}
and hence, since  $\bar{\rho}(0) =  \xi$,
\begin{align}
\|\bar{\rho}(t)\|^2 \leq \|\xi\|^2 \exp \left( \eta_0 \int_0^t \| \nabla \theta^0 \|^{2}
\, dr + (C - \lambda_N) t \right)\,.
\label{eq:bar:rho:conclus}
\end{align}
Applying Proposition \ref{prop:lets:make:the:most:of:the:moments:that:count} we conclude that,
for any $\Tho_0 \in H$, and $\eta\in(0,\eta_0]$,
\begin{align*}
\E \|\bar{\rho}(t)\|^2 \leq C \|\xi\|^2 \exp\left( \eta \|\Tho_0\|^2 +  (C + \eta - \lambda_N) t \right) \,,
\end{align*}
where $C = C(\Ra, \rab)$ is independent of $\xi$ and $\Tho_0$ and $t \geq 0$.
By now choosing $N$ large enough such that $\lambda_N > 2( C + \eta \| \sigma\|^2)$ we obtain
\begin{align}
\E \|\bar{\rho}(t)\|^2 \leq C \|\xi\|^2 \exp\left( \eta \|\Tho_0\|^2  - \frac{\lambda_N}{2} t \right) \,,
\label{eq:rhb:est}
\end{align}
where $C = C(\Ra, \rab)$ is independent of $\xi$ and $\Tho_0$ and $t \geq 0$.
This yields the first bound \eqref{eq:suff:cond:ASF:bnd:2}.

To obtain the second desired bound, \eqref{eq:suff:cond:ASF:bnd:1}, we use
\eqref{eq:rhb:est} to estimate
\begin{align*}
\E  \int_0^\infty |w(\xi)|^2 dt =& \|\sigma^{-1} \|^2 \lambda_N^2 \E  \int_0^\infty \|P_N \bar{\rho}\|^2 \, dt \leq
C  \exp( \eta \|\Tho_0\|^2) \,,
\end{align*}
where $C = C(\lambda_N, \ra, \rab)$ is independent of $\Tho_0$
yielding \eqref{eq:suff:cond:ASF:bnd:1}.  The bound \eqref{eq:ASF:type:bnd} now follows.

\begin{Rmk}
We can use the same argument leading to \eqref{eq:bar:rho:conclus} to show that
\begin{align*}
\|\rho(t)\|^2 \leq \|\xi\|^2 \exp \left( \eta \int_0^t \| \nabla \theta^0 \|^{2}
\, dr + C t \right)\,,
\end{align*}
that is, for any $\eta>0$,
\begin{align}
\|\JJ_{0,t}\| \leq \exp \left( \eta \int_0^t \| \nabla \theta^0 \|^{2}
\, dr + Ct \right)\,,
\label{eq:J:bound}
\end{align}
where, as above,  $C = C(\ra, \rab)=\frac{\tilde{C}\ra^{4}}{\eta^{2}}+ \ra\rab$.
\end{Rmk}

\begin{Rmk}
Using Proposition \ref{prop:inf:prnd:bound} and \eqref{eq:J:bound} we can easily establish
the Lyapunov bound \eqref{eq:decay:exp:norm}  with
\begin{align*}
C_1=\exp\left(\frac{C\ra^4e^{t^*/2}}{\eta^2}+\ra\rab\right).
\end{align*}
\end{Rmk}

\section*{Acknowledgments}
The majority of this work was carried out during a Research in Peace
(RIP) fellowship at the Institut Mittag-Leffler, Stockholm, Sweden.
We are grateful to this institution for providing us with a unique
working environment to continue our ongoing collaborations.  JF and GR
would like to acknowledge Virginia Tech for facilitating a research
visit during which this work was finalized.  We would also like to
extend our warm thanks to P. Constantin, C. Doering, S. Friedlander,
J. Mattingly, E. Thomann, and J. Whitehead for stimulating feedback on
this work. NEGH was partially supported in this work under the grants
NSF-DMS-1313272, Simons-515990.

\addcontentsline{toc}{section}{References}
\begin{footnotesize}
\bibliographystyle{plain}

\end{footnotesize}

\newpage

\vspace{.3in}
\begin{multicols}{2}
\noindent
Juraj F\"oldes\\
{\footnotesize Department of Mathematics\\
University of Virginia\\
Web: \url{http://www.people.virginia.edu/~jf8dc/}\\
 Email: \url{foldes@virginia.edu}} \\[.25cm]
Nathan Glatt-Holtz\\ {\footnotesize
Department of Mathematics\\
Tulane University\\
Web: \url{http://www.math.tulane.edu/~negh/}\\
 Email: \url{negh@tulane.edu}} \\[.2cm]

\columnbreak

 \noindent Geordie Richards\\
{\footnotesize
Department of Mechanical and Aerospace Engineering\\
Utah State University\\
Web: \url{www.geordierichards.com}\\
 Email: \url{geordie.richards@usu.edu}}\\[.2cm]
 \end{multicols}

\end{document}